\documentclass[reqno,11pt,a4paper]{amsart}
\usepackage[draft=false]{hyperref}
\usepackage{microtype}
\usepackage{url}
\usepackage{amsmath,amssymb,amsfonts}
   \DeclareMathOperator{\Id}{Id}
   \DeclareMathOperator{\e}{e}
   \DeclareMathOperator{\Lip}{Lip}
\usepackage{amsthm}
   \newtheorem{theorem}{Theorem}[section]
   \newtheorem{lemma}[theorem]{Lemma}
   \newtheorem{proposition}[theorem]{Proposition}
   \newtheorem{example}[theorem]{Example}
   \newtheorem{corollary}[theorem]{Corollary}
\usepackage{dsfont,mathrsfs}
   \newcommand{\N}{\mathds{N}}
   \newcommand{\Z}{\mathds{Z}}
   \newcommand{\R}{\mathds{R}}
   \newcommand{\T}{\mathbb{T}}
   \newcommand{\cA}{\mathcal{A}}
   \newcommand{\cB}{\mathcal{B}}
   \newcommand{\sF}{\mathscr{F}}
   \newcommand{\cG}{\mathcal{G}}
   \newcommand{\cH}{\mathcal{H}}
   \newcommand{\cV}{\mathcal{V}}
   \newcommand{\cX}{\mathfrak{X}}
   \newcommand{\cL}{\mathfrak{L}}
   \newcommand{\cJ}{\mathfrak{J}}
\usepackage{enumitem}
   \newcommand{\lb}{label}
   \newcommand{\lm}{leftmargin}

\newcommand{\eps}{\varepsilon}
\newcommand{\lbd}{\lambda}
\newcommand{\sgm}{\sigma}
\newcommand{\Sgm}{\Sigma}
\newcommand{\bxi}{\,\overline{\!\xi}}
\newcommand{\w}{\omega}
\newcommand{\W}{\Omega}
\newcommand{\phiw}[1]{\Phi_{\theta^{#1}\w}^{-{#1}}}
\newcommand{\phii}[2]{\Phi_{\theta^{#1}\w}^{-{#2}}}

\newcommand{\prts}[1]{\left(#1\right)}
\newcommand{\prtsr}[1]{\left[#1\right]}
\newcommand{\abs}[1]{\left|#1\right|}
\newcommand{\norm}[1]{\left\|#1\right\|}
\newcommand{\set}[1]{\left\{#1\right\}}
\newcommand{\setm}[1]{\setminus\set{#1}}
\newcommand{\maxs}[1]{\max\set{#1}}
\newcommand{\mins}[1]{\min\set{#1}}
\newcommand{\sups}[1]{\sup\set{#1}}

\newcommand{\x}{\times}

\newcommand{\pfrac}[2]{\prts{\dfrac{#1}{#2}}}

\newcommand{\dsum}{\displaystyle\sum}
\newcommand{\dint}{\displaystyle\int}

\newcommand{\dlim}{\displaystyle\lim}
\newcommand{\dr}{\, dr}
\newcommand{\ds}{\, ds}

\newcommand{\dmu}{\, d\mu}
\renewcommand{\ge}{\geqslant}

\renewcommand{\le}{\leqslant}
\renewcommand{\leq}{\leqslant}
\begin{document}
\title[Invariant manifolds for RDS on Banach Spaces]
   {Invariant manifolds for Random Dynamical Systems on Banach Spaces exhibiting generalized dichotomies}
\author[Ant\'onio J. G. Bento]{Ant\'onio J. G. Bento}
   \address{
      Ant\'onio J. G. Bento\\
      Departamento de Matem\'atica\\
      Universidade da Beira Interior\\
      6201-001 Covilh\~a\\
      Portugal}
   \address{
      Centro de Matem\'atica e Aplica\c{c}\~oes\\
      Universidade da Beira Interior\\
      6201-001 Covilh\~a\\
      Portugal}
   \email{bento@ubi.pt}
\author{Helder Vilarinho}
   \address{
      Helder Vilarinho\\
      Departamento de Matem\'atica\\
      Universidade da Beira Interior\\
      6201-001 Covilh\~a\\
      Portugal}
   \address{
      Centro de Matem\'atica e Aplica\c{c}\~oes\\
      Universidade da Beira Interior\\
      6201-001 Covilh\~a\\
      Portugal}
   \email{helder@ubi.pt}
   \urladdr{www.mat.ubi.pt/~helder}
\date{\today}
\begin{abstract}
   We prove the existence of measurable invariant manifolds for small perturbations of linear Random Dynamical Systems evolving on a Banach space and admitting a general type of dichotomy, both for continuous and discrete time. Moreover, the asymptotic behavior in the invariant manifold is similar to the one of the linear Random Dynamical System.
\end{abstract}
\maketitle
\section{Introduction}

One of the main issues in Dynamical Systems is the study of properties and structures (geometric, topological, ergodic, ...)  that are invariant over time, either in the deterministic or in the random evolutionary systems. The study of invariant manifolds for deterministic dynamical systems goes back to the works of Hadamard~\cite{Hadamard-BSMF-1901}, Lyapunov~\cite{Lyapunov-IJC-1992} and Perron~\cite{Perron-MZ-1929,Perron-JFRAM-1929,Perron-MZ-1930}. For an historical background see for example \cite{Bates_Lu_Zeng-book-1998}. In the Random Dynamical Systems (RDS) framework there are several works covering (local and/or global) center, stable, unstable and inertial manifolds for a variety of state spaces, that goes from the Euclidean space to Hilbert spaces or separable Banach spaces, either generated by stochastic or by random differential equations. The list of works on this subject is already too extensive to be completely written down here. We refer for \cite{Arnold-RDS-1998, Liu_Qian-book-1995, Lian_Lu-book-2010, Wanner-proc-1995, Mohammed_Zhang_Zhao-book-2008}. See also \cite{Barreira_Valls-SD-2018,Caraballo_Duan_Lu_Schmalfuss-ANS-2010, Ruelle-AM-1982,Lu_Schmalfuss-proc-2012} and references therein. For invariant manifolds of RDS on infinite dimensional Banach space see \cite{Bates_Lu_Zeng-book-1998,Duan_Lu_Schmalfuss-AP-2003,Lian_Lu-book-2010,Barreira_Valls-SD-2018}.

In this work we prove the existence of random global invariant manifolds for RDS evolving on a Banach space (not necessarily separable), both in the continuous and in the discrete time settings. The RDS considered are obtained by perturbing linear RDS that admit a generalized dichotomy. In the deterministic cases this kind of dichotomies were considered in \cite{Bento_Silva-BSM-2014,Bento_Silva-PM-2016} and generalize the common nonuniform exponential condition that often arises from the Multiplicative Ergodic Theorem. To the best of our knowledge this kind of dichotomies were not considered before in the RDS setting. Moreover, the perturbations considered in this work satisfy some natural conditions that guarantee not only the existence of invariant manifolds but also some control on the dynamics.

We notice that the separability of the state space is not assumed, which in the continuous time case requires some special attention to measurability and integrability issues. Moreover, for our purposes, the driving system consists of an invertible dynamical system defined on a measure space that is not necessarily finite, as typically considered on RDS theory.

In the RDS setting one expects for properties that hold for almost every element in the driving system. Throughout this work we will assume that it is possible to consider properties that hold for all elements of the driving system by restricting this dynamics, if necessary, to a full measure invariant subset (see Lemma~\ref{lemma:restricted_DS}).

The main strategy used to obtain the random invariant manifolds follows the Lyapunov-Perron approach. We define a convenient space of pairs of functions that is a complete metric space and use the Banach Fixed Point Theorem to obtain the invariant manifolds as the graph of a function. We notice that no random norms were considered and that the determinist cases can be easily deduced from the random counterpart.  We also give several examples that include the usual nonuniform exponential situation and also that illustrates other situations beyond this behaviour.

The paper is organized as follows. In Section~\ref{Notation and preliminaries} we recall the notion of the Bochner integral and give an elementary introduction to RDS. Moreover, we also define generalized dichotomies, to which we give some examples, and the type of perturbations considered. In Section~\ref{sec:cont_time} we state the main theorem for continuous time (Theorem~\ref{thm:global}) and get some corollaries. The section finishes with the proof of Theorem~\ref{thm:global}. The discrete time case is discussed in Section~\ref{sec:disc_time}, where we state the main theorem (Theorem~\ref{thm:global:disc}), give some corollaries and give its proof.

\section{Notation and preliminaries}\label{Notation and preliminaries}
\subsection{Bochner integral}

We start by compiling some facts about the Boch\-ner integral. Let $(A,\cA)$ and $(B,\cB)$ be measurable spaces and let $X$ be a Banach space.
A map $g \colon A \to B$ is  \emph{$(\cA,\cB)$-measurable} if $g^{-1}(U) \in \cA$ for every $U \in \cB$ and a map $h \colon A \to X$ is \emph{simple} if there are distinct elements $y_1, \ldots, y_n \in X$ and pairwise disjoint sets $A_1, \ldots, A_n \in \cA$ such that ${A_1 \cup \cdots \cup A_n = A}$ and
   $$ h(a) = \dsum_{i=1}^n y_i \cdot \chi_{A_i}(a),$$
where $\chi_{A_i}$ is the indicator function of $A_i$. A function $h \colon A \to X$  is \emph{Bochner measurable} if there is a sequence of simple functions $h_n \colon A \to X$ such that
\begin{equation*}
   \lim_{n \to +\infty} \|h_n(a) - h(a)\| = 0 \ \ \ \text{ for every } a \in A.
\end{equation*}
This property is sometimes called strong measurability, which, in turn, is also used with a different meaning as we remark below. In order to avoid some misunderstandings we will always use the expression Bochner measurable, that is motivated by the goal of using Bochner integrals.

Given a topological space $T$, we denote by $\cB(T)$ the $\sigma$-algebra generated by the open subsets of $T$.

\begin{proposition}[{\cite[Corollary 1.1.10]{Hytonen_Neerven_Veraar_Weis-book-2016}}] \label{prop:Borel_x_Bochner}
   Let $(A,\cA)$ be a measurable space, let $X$ be a Banach space and consider $h \colon A \to X$. Then $h$ is Bochner measurable if and only if it is $(\cA,\cB(X))$-measu\-ra\-ble and has separable range.
\end{proposition}

An immediate consequence of the above proposition is the following corollary.

\begin{corollary} \label{cor:Bmeas_o_meas_is_Bmeas}
   Let $(A,\cA)$ and $(B,\cB)$ be measurable spaces, let $X$ be a Banach space and consider maps $g \colon A \to B$ and $h \colon B \to X$.
   If $g$ is $(\cA,\cB)$-measurable and $h$ is Bochner measurable, then $h\circ g$ is Bochner measurable.
\end{corollary}

Let $(A,\cA,\mu)$ be a measure space. We say that a simple function
\begin{equation*}
   h(a) = \sum_{i=1}^n y_i \cdot \chi_{A_i}(a)
\end{equation*}
is Bochner integrable if $\mu(A_i) < +\infty$ for all $i=1, \ldots,n$ such that $y_i \ne 0$, and its integral is given by
\begin{equation*}
   \dint_A h \dmu = \dsum_{i=1}^n y_i \cdot \mu(A_i),
\end{equation*}
with the convention $0 \times (+\infty) = 0$. A Bochner measurable function $h \colon A \to X$ is \textit{Bochner integrable} if there is a sequence $h_n \colon A \to X$ of simple Bochner integrable functions pointwise convergent to $h$ and such that
\begin{equation*}
   \lim_{n \to +\infty} \dint_A ||h_n - h|| \dmu = 0,
\end{equation*}
where the integral considered here is the Lebesgue integral. Then the sequence $\prts{\int_A h_n \dmu}_{n \in \N}$ is convergent (in $X$) and the Bochner integral of $h$ is given by
\begin{equation*}
   \dint_A h \dmu
   = \lim_{n \to +\infty} \dint_A h_n \dmu.
\end{equation*}

\begin{proposition}[{\cite[Proposition 1.2.2]{Hytonen_Neerven_Veraar_Weis-book-2016}}]\label{prop:Bochner:2}
   Let $(A,\cA,\mu)$ be a measure space, let $X$ be a Banach space and consider a Bochner measurable map $h \colon A \to X$. Then $h$ is Bochner integrable if and only if $\|h\|$ is Lebesgue integrable. Moreover, if $h$ is Bochner integrable we have
   \begin{equation*}
      \norm{\dint_A h \dmu}
      \le \dint_A \norm{h} \dmu.
   \end{equation*}
\end{proposition}

\subsection{Random Dynamical Systems on Banach spaces}
Let us recall now some basic concepts on Random Dynamical Systems (RDS). For a complete introduction we recommend \cite{Arnold-RDS-1998}.
Let $\T$ be $\R$ or $\Z$ depending if we are concerned in the continuous or in the discrete time, respectively, and let $\T^+ = \T \cap [0,+\infty[$. Consider a measure space $(\W, \Sgm, \mu)$ and a \emph{measure-preserving dynamical system} $(\W, \Sgm, \mu, \theta)$, in the sense that
\begin{enumerate}[\lb= ,\lm=10mm]
   \item $\theta \colon \T \times \W \to \W$ is $\prts{\cB(\T) \otimes \Sgm ,\Sgm}$-measurable;
   \item $\theta^t \colon \W \to \W$ given by $\theta^t \w = \theta(t,\w)$ preserves the measure $\mu$ for all $t \in \T$;
   \item $\theta^0 = \Id_{\W}$ and $\theta^{t+s}=\theta^t\circ\theta^s$ for
       all $t,s \in\T$.
\end{enumerate}
If $\mu$ is a probability measure, then $(\W, \Sgm, \mu, \theta)$ is called a \emph{metric dynamical system}.
A (\emph{measurable}) \emph{random dynamical system} (RDS) on a Banach space $X$ over
$(\W, \Sgm, \mu, \theta)$ with time $\T^+$ is a map
\begin{equation*}
  \Phi:\T^+\times\W\times X \to X
\end{equation*}
such that
\begin{enumerate}[\lb=$\roman*)$,\lm=10mm]
   \item $(t,\w) \mapsto \Phi(t,\w,x)$ is $\prts{\cB(\T^+) \otimes \Sgm,\cB(X)}$-measurable for every $x \in X$;
   \item $\Phi_\w^t \colon X \to X$ given by $\Phi_\w^t x = \Phi(t,\w,x)$ forms a cocycle over $\theta$, i.e.,
      \begin{enumerate}[\lb=$\alph*)$,\lm=6mm]
         \item $\Phi_\w^0=\Id_X$ for all $\w\in\W$;
         \item $\Phi_\w^{t+s}=\Phi_{\theta^s\w}^t\circ\Phi_\w^s$, for all $s, t\in\T^+$ and $\w\in\W$.
      \end{enumerate}
\end{enumerate}

If, in addition,
\begin{enumerate}[\lb=$\roman*')$,\lm=10mm]
   \item $(t,\w) \mapsto \Phi(t,\w,x)$ is Bochner measurable for every $x \in X$
\end{enumerate}
we say that $\Phi$ is a \textit{Bochner measurable RDS}. Property $i)$ is also called strong measurability (for $\Phi$). By Proposition~\ref{prop:Borel_x_Bochner}, conditions $i)$ and $i')$ are equivalent when $X$ is separable.

A measurable RDS $\Phi$ is called \emph{linear} if $\Phi_\w^t$ is a bounded linear operator for each $\w\in\W$ and $t\in\T^+$.

In RDS theory it is typically assumed that the driving system $(\W, \Sgm, \mu, \theta)$ is a metric dynamical system. However, for our purposes we do not assume \emph{a priori} that the measure $\mu$ is finite.

\begin{lemma}\label{lemma:restricted_DS}
   Let $\Theta\equiv(\W,\Sgm,\mu,\theta)$ be a measure preserving dynamical system and consider a measurable set $\ \W'\in\Sgm$ that is $\theta^t$-invariant for all $t \in \T$. Let $\Sgm'=\{B\cap\W'\colon B\in\Sgm\}$ be the trace of $\Sgm$ with respect to $\W'$, and let $\mu'=\mu\vert_{\W'}$ and $\theta'=\theta\vert_{\W'}$. Then:
   \begin{enumerate}[\lb=$\roman*)$,\lm=10mm]
      \item $\Theta'\equiv(\W',\Sgm',\mu',\theta')$ is a measure preserving dynamical system;
      \item if $\Phi$ is a measurable (resp. Bochner measurable) RDS over $\Theta$ then $\Phi\vert_{\T^+\x\W'\x X}$ is a measurable (resp. Bochner measurable) RDS over $\Theta'$.
   \end{enumerate}
\end{lemma}

\begin{proof}
   The first item is proved in~\cite[Lemma 3.2]{Caraballo_Duan_Lu_Schmalfuss-ANS-2010} and the proof of the second item is analogous. The Bochner measurable RDS case follows from Proposition~\ref{prop:Borel_x_Bochner}.
\end{proof}

\subsection{Generalized Dichotomies}

Given a map $P \colon \W \x X \to X$, we say that a measurable (resp. Bochner measurable) linear RDS $\Phi$ admits a measurable (resp. Bochner measurable) \textit{P-invariant splitting} if
\begin{enumerate}[\lb=$\roman*)$,\lm=10mm]
   \item $\w \mapsto P(\w,x)$ is $\prts{\Sgm,\cB(X)}$-measurable (resp. Bochner measurable) for every $x \in X$;
   \item $P_\w \colon X \to X$ defined by $P_\w x = P(\w,x)$ is a linear bounded projection for all $\w \in \W$;
   \item $P_{\theta^t\w} \Phi_{\w}^t = \Phi_{\w}^t P_\w$ for all $t \in \T^+$ and all $\w \in \W$;
   \item $\Phi_{\w}^t(\ker P_\w) = \ker P_{\theta^t\w}$ for all $t \in \T^+$ and all $\w \in \W$;
   \item $\Phi_{\w}^t|_{\ker P_\w} \colon \ker P_\w \to \ker
      P_{\theta^t\w}$ is invertible for all $t \in \T^+$ and all $\w \in \W$;
   \item setting $Q_{\w} = \Id-P_{\w}$, the map
      \begin{equation*}
        (t,\w) \mapsto (\Phi_{\w}^t|_{\ker P_\w})^{-1}Q_{\theta^t \w}x
      \end{equation*}
      is $\prts{\cB(\T^+) \otimes \Sgm,\cB(X)}$-measurable (resp. Bochner measurable) for every $x \in X$.
\end{enumerate}
In order to simplify the notation we will denote by $\phiw{t}$ the inverse of
   $$\Phi_{\w}^t|_{\ker P_\w} \colon \ker P_\w \to \ker P_{\theta^t\w}.$$
In these conditions we define the linear subspaces $E_\w =P_\w (X)$ and $F_\w= \ker P_\w = Q_\w(X)$ and, as
usual, we identify the vector spaces $E_\w \times F_\w$ and $E_\w \oplus F_\w$
as the same vector space.

Given functions $\alpha^+, \alpha^-:\T^+ \x \W  \to \, ]0,+\infty[$, and
denoting $\alpha^+(t,\w)$ and $\alpha^-(t,\w)$ by $\alpha^+_{t,\w}$ and $\alpha^-_{t,\w}$,
respectively, we say that a measurable (resp. Bochner measurable) linear RDS $\Phi$ admits a \textit{generalized dichotomy with bounds $\alpha^+$ and $\alpha^-$} if it admits a measurable (resp. Bochner measurable) $P$-invariant splitting such that
\begin{enumerate}[\lb=$($D$\arabic*)$,\lm=13mm]
   \item \label{eq:split4} $\|\Phi_{\w}^t P_\w\| \le \alpha^+_{t,\w}$ for all $(t,\w) \in \T^+ \times \W$;
   \item \label{eq:split5} $\|\phiw{t} Q_{\theta^t\w}\| \le
      \alpha^-_{t,\theta^t\w}$  for all $(t,\w) \in \T^+ \times \W$.
\end{enumerate}

The following example corresponds to the usual tempered exponential dichotomies.
We recall that a random variable $K\colon\W\to[1,+\infty[$ is \emph{tempered} if
\begin{equation}\label{eq:tempered}
    \lbd_{K,\gamma,\w}:=\sup_{t \in \T} \prtsr{\e^{-\gamma |t|}  K(\theta^t w)} < +\infty
\end{equation}
for all $\gamma >0$ and all $\w \in \W$. We notice that a weaker condition for tempered random variables is often used:
    \begin{equation}\label{eq:tempered:lim}
             \lim_{t\to\pm\infty}\frac1{\abs{t}} \log K(\theta^t\w)=0
             \text{ for all $\w\in\W$}.
    \end{equation}
In the discrete time case the conditions are equivalent, however in the continuous time situation we will use~\eqref{eq:tempered} in order to deal with the computations in the proof of Corollary~\ref{corollary:tempered:cont}.

\begin{example}[(Non)uniformly (pseudo-)hyperbolic]
   \label{ex:nonuniformly:pseudohyperbolic}
   Let $\Theta \equiv (\W,\Sgm,\mu,\theta)$ be a metric dynamical system and let $X$ be a Banach space. We say that a measurable linear RDS $\Phi$ on $X$ over $\Theta$ admits a \emph{tempered exponential dichotomy} if it admits a generalized dichotomy with bounds
   \begin{equation*}
    \alpha^+_{t,\w} = K(\w) \e^{a(\w)t}
    \ \ \ \text{ and  } \ \ \
    \alpha^-_{t,\theta^t\w} = K(\theta^t\w) \e^{b(\w)t},
   \end{equation*}
   for some tempered random variable $K:\W \to [1,+\infty[$ and $\theta$-invariant random variables $a, b \colon \W \to \R$, i.e. , satisfying $a(\theta^t \w) = a(\w)$ and $b(\theta^t \w) = b(\w)$ for all $\w\in\W$ and all $t\in\T^+$.
   Tempered dichotomies are of particular interest since they can be obtained throughout Oseledets' Multiplicative Ergodic Theorem. See \cite[Theorem 3.5]{Lian_Lu-book-2010} and \cite[Theorem 3.4]{Caraballo_Duan_Lu_Schmalfuss-ANS-2010}. The common situations occurs when $a, b$ and $K$ are constant (uniformly hyperbolic), $a(\w)=b(\w)<0$ (nonuniformly hyperbolic) or $a(\w)+b(\w)<0$ (nonuniformly pseudo-hy\-per\-bo\-lic).
\end{example}

We give now another example of dichotomies that illustrate situations far beyond the exponential growth rates.

\begin{example}\label{ex:2}
   Let $\Theta\equiv(\W,\Sgm,\mu,\theta)$ be a measure-preserving dynamical system. Consider measurable functions
      $$ \varphi, \psi  \colon \T^+ \x \W \to ]0,+\infty[$$
   such that
   \begin{equation*}
      \varphi(t+s,\w) = \varphi(t,\theta^s \w) \varphi (s,\w)
      \ \ \ \text{ and } \ \ \
      \psi(t+s,\w) = \psi(t,\theta^s \w) \psi (s,\w)
   \end{equation*}
   for all $t,s \in \T^+$ and all $\w \in \W$.
   Let $X=\R^2$, equipped with the maximum norm, let $K \colon \W \to [1,+\infty[$ be a random variable and consider the complementary projections $P_\w, Q_\w \colon \R^2 \to \R^2$ given by
   \begin{equation*}
      P_\w(x_1,x_2)=(x_1+(K(\w)-1)x_2, 0)
      \ \ \ \text{ and }   \ \ \
      Q_\w(x_1,x_2)=((1-K(\w))x_2, x_2).
   \end{equation*}
   It is easy to see that
   \begin{equation*}
      P_{\overline \w} P_\w = P_\w, \ \ \
      Q_{\overline \w} Q_\w = Q_{\overline \w}, \ \ \
      P_\w Q_\w = 0 \ \ \ \text{ and } \ \ \
      Q_{\overline \w} P_\w = 0
   \end{equation*}
   for all $\w, {\overline \w} \in \W$, and these equalities imply that $\Phi\colon\T^+\x\W\x \R^2\to\R^2$ defined by
   \begin{equation*}
      \Phi_\w^t
      = \varphi(t,\w) \, P_\w
         + \dfrac{K(\w)}{K(\theta^t \w)} \dfrac{1}{\psi(t,w)} \, Q_{\theta^t \w}
   \end{equation*}
   is a measurable linear RDS over $\Theta$ that admits a measurable $P$-invariant splitting. Moreover,
   \begin{equation*}
      \|\Phi^t_\w P_\w\|
      = \varphi(t,\w) \|P_\w\|
      = K(\w) \varphi(t,w),
   \end{equation*}
   and since $\phiw{t} Q_{\theta^t \w} = \dfrac{K(\theta^t \w)}{K(\w)} \psi(t,\w) Q_\w$ and $\|Q_\w\| = \maxs{K(\w)-1,1} \le K(\w)$ we have
   \begin{equation*}
      \|\phiw{t} Q_{\theta^t w}\|
      = \dfrac{K(\theta^t \w)}{K(\w)} \psi(t,\w) \|Q_\w\|
      \le K(\theta^t \w) \psi(t,\w).
   \end{equation*}
   Hence the linear RDS $\Phi$ admits a generalized dichotomy with bounds
   \begin{equation*}
      \alpha^+_{t,\w} = K(\w) \varphi(t,\w)
      \ \ \ \text{ and } \ \ \
      \alpha^-_{t,\theta^t \w} = K(\theta^t \w) \psi(t,\w).
   \end{equation*}

   An interesting case occurs when we consider $\T=\R$ and random variables $a,b: \W \to \R$ such that for all $\w$ the maps $s \mapsto a(\theta^s \w)$ and $s \mapsto b(\theta^s \w)$ are integrable in every interval $[0,t]$, $t\ge 0$,
   and make
   \begin{equation*}
      \varphi(t,\w) = \e^{\int_0^t a(\theta^s \w) \ds}
      \ \ \ \text{ and } \ \ \
      \psi(t,\w) = \e^{\int_0^t b(\theta^s \w) \ds}.
   \end{equation*}
   If we assume that $\Theta$ is a metric dynamical system, $K$ is a tempered random variable and letting $a,b:\W\to\R$ to be $\theta$-invariant random variables, we get tempered exponential dichotomies for this particular case. When $\T=\Z$ for the analogous case we take
   \begin{equation*}
      \varphi(n,\w) = \e^{S_a(n,\w)}
      \ \ \ \text{ and } \ \ \
      \psi(n,\w) = \e^{S_b(n,\w)},
   \end{equation*}
   where
   \begin{equation*}
     S_Z(n,\w)=\sum_{r=0}^{n-1} Z(\theta^r \w)
   \end{equation*}
   for a given random variable $Z\colon\W\to\R$.

   Another case occurs when we consider
   \begin{equation*}
      \varphi(t,\w)=\frac{a(\w)}{a(\theta^t \w)}
      \ \ \ \text{ and } \ \ \
      \psi(t,\w)=\frac{b(\w)}{b(\theta^t \w)},
   \end{equation*}
   where $a,b \colon \W \to \, ]0,+\infty[$ are random variable, which gives dichotomies with bounds
   \begin{equation*}
      \alpha^+_{t,\w} =  K(\w) \frac{a(\w)}{a(\theta^t \w)}
      \ \ \ \text{ and } \ \ \
      \alpha^-_{t,\theta^t \w} = K(\theta^t \w)\frac{b(\w)}{b(\theta^t \w)}
   \end{equation*}
   that can lead us to nonexponential growth rates. To see this take for the driving system the horizontal flow in $\R^2$ given by $\,\theta^t (x,y) = (x+t,y)$, which preserves the Lebesgue measure, and set:
   \begin{align*}
      & a(x,y)=(1+x^2)^{-\lambda(1+y^2)}\\
      & b(x,y)=(1+x^2)^{-\gamma(1+y^2)}\\
      & K(x,y)=C (1+x^2)^{\eps(1+y^2)},
   \end{align*}
   for some real constants $C, \lambda, \gamma, \eps$, with $C \ge 1$ and  $\eps \ge 0$. In this case we obtain a polynomial type dichotomy with bounds
      $$ \alpha^+_{t,(x,y)}
         = C\pfrac{1+(x+t)^2}{1+x^2}^{\lbd(1+y^2)} (1+x^2)^{\eps(1+y^2)}$$
   and
      $$ \alpha^-_{t,\theta^t (x,y)}
         =C \pfrac{1+(x+t)^2}{1+x^2}^{\gamma(1+y^2)} (1+(x+t)^2)^{\eps(1+y^2)}.$$
\end{example}

\subsection{Auxiliary spaces of functions}

Consider a measure preserving dynamical system $\Theta \equiv (\W, \Sgm, \mu, \theta)$, a Banach space $X$ and a measurable linear RDS $\Phi$ on $X$ over $\Theta$ that admits a dichotomy with bounds $\alpha^+$ and $\alpha^-$.

Let $\sF$ be the space of all functions $f \colon \W \x X \to X$ such that, denoting $f(\w,x)$ by $f_\w(x)$, satisfy
\begin{align}
   & \label{eq:w|->f_w(x):Borel}
      \w \mapsto f_\w(x) \text{ is $(\Sgm,\cB(X))$-measurable for every } x \in X;
\end{align}
and, for every $\w \in \W$,
\begin{align}
   & \label{eq:f_w(0)=0}
      f_\w(0)=0;\\
   & \label{eq:Lip(f_w)}
      \Lip(f_\w) = \sup\set{\dfrac{\|f_\w(x) - f_\w(y)\|}{\|x - y\|}
         \colon x,y \in X,\ x \ne y}<+\infty.
\end{align}
Clearly, from~\eqref{eq:Lip(f_w)} and~\eqref{eq:f_w(0)=0} we have for all $\w \in \W$ and all $x, y \in X$ that
\begin{align}
   & \label{ine:||f_w(x)-f_w(y)||<=Lip(f_w)||x-y||}
      \norm{f_\w(x) - f_\w(y)} \le \Lip(f_\w) \|x - y\|;\\
   & \label{ine:||f_w(x)||<=Lip(f_w)||x||}
      \norm{f_\w(x)} \le \Lip(f_\w) \|x\|.
\end{align}

We denote by $\sF^{(B)}$ the space of all functions $f \colon \W \x X \to X$ that satisfy~\eqref{eq:f_w(0)=0},~\eqref{eq:Lip(f_w)} and
\begin{equation}\label{eq:w|->f_w(x):Bochner}
   \w \mapsto f_\w(x) \text{ is Bochner measurable for every } x \in X.
\end{equation}
From Proposition~\ref{prop:Borel_x_Bochner} it is clear that~\eqref{eq:w|->f_w(x):Bochner} implies~\eqref{eq:w|->f_w(x):Borel} and thus $\sF^{(B)} \subseteq \sF$.

Set $\alpha=(\alpha^+,\alpha^-)$ and denote by  $\sF^{(B)}_{\alpha}$ the space of all functions $f\in\sF^{(B)}$ such that, for all $\w \in \W$, the maps
   \begin{equation}\label{eq:s|->a+_t-s,ths_w...:Lebesgue}
   s \mapsto \alpha^+_{t-s,\theta^s \w} \Lip(f_{\theta^s \w}) \alpha^+_{s,\w}
   \ \ \ \text{ and } \ \ \
   s \mapsto \alpha^-_{s,\theta^s \w} \Lip(f_{\theta^s \w}) \alpha^+_{s,\w}
\end{equation}
are measurable on every interval $[0,t]$, $t\ge 0$.

Let
   $$ \cH =
      \set{(t,\w,\xi)\in\T^+\x \W\x X \colon
      \xi \in E_{\w}}\subset\T^+\x\W\x X.$$
Given $M > 0$, denote by $\cJ_M$ the space of all functions $h : \cH \to X$ such that, writing $h_{t,\w}(\xi)$ for $h(t,\w,\xi)$, satisfy
\begin{align}
   & \label{eq:(t,w)|->h_t,w(xi):Borel}
      (t,\w) \mapsto h_{t,w}(P_\w x)
      \text{ is  $(\cB(\T^+) \otimes \Sgm,\cB(X))$-measurable for all } x \in X;\\
   & \label{eq:h_n,w(0)=0}
      h_{t,\w}(0) =0
      \text{ for every } (t,\w) \in \T^+ \x \W;\\
   & \label{eq:h_0,w(xi)=xi}
      h_{0,\w}(\xi)=\xi
      \text{ for every } \w \in \W \text{ and every } \xi \in E_\w;\\
   & \label{h_n,w(xi):in:E(t^n(w))}
      h_{t,\w}(\xi) \in E_{\theta^t\w}
      \text{ for all } (t,\w, \xi) \in \cH;\\
   & \label{ine:||h_n,w(xi)-h_n,w(bxi)||<=M...}
      \|h_{t,\w}(\xi)-h_{t,\w}(\bar\xi)\| \le M \|\xi-\bar\xi\| \alpha^+_{t,\w}
      \text{ for all } (t,\w,\xi), (t,\w,\bxi) \in \cH.
\end{align}
From~\eqref{ine:||h_n,w(xi)-h_n,w(bxi)||<=M...} and~\eqref{eq:h_n,w(0)=0} it follows immediately that
\begin{equation}\label{ine:||h_n,w(xi)||<=M...}
   \|h_{t,\w}(\xi)\| \le M \|\xi\|\alpha^+_{t,\w}
      \text{ for all } (t,\w,\xi) \in \cH.
\end{equation}

It is straightforward that ${\cJ_M}$ equipped with the metric
\begin{equation} \label{def:d_1(x,y)}
   d_1(h,g)
   = \sup\set{\dfrac{\|h_{t,\w}(\xi)-g_{t,\w}(\xi)\|}{\alpha^+_{t,\w} \|\xi\|} \colon
      (t,\w) \in \T^+ \x \W, \ \xi \in E_\w \setm{0}}
\end{equation}
is a complete metric space. If in the definition of $\cJ_M$ we replace condition~\eqref{eq:(t,w)|->h_t,w(xi):Borel} by
\begin{equation}\label{eq:(t,w)|->h_t,w(xi):Bochner}
   (t,\w) \mapsto h_{t,w}(P_\w x)
   \text{ is  Bochner measurable for all } x \in X.
\end{equation}
we obtain a complete metric subspace of $\cJ_M$ that we denote by $\cJ^{(B)}_M$.

Let
\begin{equation*}
   \cG=\{(\w,\xi)\in\W\x X\colon \xi\in E_{\w}\}\subset\W\x X.
\end{equation*}
Given $N > 0$, we denote by $\cL_N$ the space of all functions $\phi: \cG \to X$ that, writing $\phi_\w(\xi)$ for $\phi(\w,\xi)$, satisfy the following conditions
\begin{align}
   & \label{eq:w|->phi_w(xi):Borel}
      \w \mapsto \phi_\w(P_\w x)
      \text{ is $(\Sgm,\cB(X))$-measurable for every } x \in X;\\
   &
      \phi_\w(0)=0 \text{ for every } \w \in \W; \label{eq:phi_w(0)=0} \\
   & \label{eq:phi_w(xi):in:F_w}
      \phi_\w(\xi) \in F_\w \text{ for every } (\w,\xi) \in \cG; \\
   & \label{eq:||phi_w(xi)-pi_w(bxi)||<=...}
      \|\phi_\w(\xi) - \phi_\w(\bxi)\| \le N \|\xi - \bxi\|
      \text{ for every } (\w,\xi), (\w, \bxi) \in \cG.
\end{align}
Making $\bar\xi = 0$ in~\eqref{eq:||phi_w(xi)-pi_w(bxi)||<=...}, by~\eqref{eq:phi_w(0)=0}, we have
\begin{equation} \label{ine:||phi_w(xi)||<=N||xi||}
   \|\phi_\w(\xi)\| \le N \|\xi\|
   \text{ for every  } (\w, \xi) \in \cG.
\end{equation}

We define a metric in $\cL_N$ by
\begin{equation}\label{def:d_2(phi,psi)}
   d_2(\phi,\psi)
   = \sups{\dfrac{\|\phi_\w(\xi) - \psi_\w(\xi)\|}{\|\xi\|}
      \colon \xi \in E_\w \setm{0}, \ \w \in \W}
\end{equation}
and notice that $(\cL_N,d_2)$ is a complete metric space.

As before, if in the definition of the space $\cL_N$ we replace~\eqref{eq:w|->phi_w(xi):Borel} by the stronger condition
\begin{equation}\label{eq:w|->phi_w(xi):Bochner}
   \w \mapsto \phi_\w(P_\w x)
   \text{ is Bochner measurable for every } x \in X
\end{equation}
we obtain a complete metric subspace of $\cL_N$ that we denote by $\cL_N^{(B)}$.

Denote now by $\cX_{M,N}$ and  $\cX^{(B)}_{M,N}$  the spaces $\cJ_M \times
\cL_N$ and $\cJ^{(B)}_M \x \cL^{(B)}_N$, respectively. It is obvious that $\cX_{M,N}$ and $\cX^{(B)}_{M,N}$ equipped with the metric
   $$ d\prts{(h,\phi),(g,\psi)} = d_1(h,g) + d_2(\phi,\psi)$$
are complete metric spaces.

\section{Continuous time} \label{sec:cont_time}

Throughout this section we consider $\T = \R$ and set $\R_0^+=\T^+$.

\subsection{Main theorem (continuous time)}\label{section:main:cont}

Consider a Bochner measurable linear RDS $\Phi$ on a Banach space $X$ over a measure-preserving dynamical system $\Theta \equiv (\W,\Sgm,\mu,\theta)$ that admits a generalized dichotomy with bounds $\alpha^+$ and $\alpha^-$ and let $f \in \sF^{(B)}_\alpha$. We will be interested on maps $\Psi\colon \R_0^+ \x \W \x X \to X$ such that $(t,\w) \mapsto \Psi(t,\w,x)$ is Bochner measurable for all $x \in X$,
\begin{equation}\label{eq:Psi:cont}
   \Psi_\w^t x
   = \Phi_\w^t x
      + \dint_0^t \Phi_{\theta^{s}\w}^{t-s} f_{\theta^s \w}(\Psi_\w^s x)\,ds
\end{equation}
for all $t\in\T^+$, $x \in X$ and $\w \in \W$. We shall always assume that $\Psi(\cdot,\w,x)$ is the unique solution of equation
\begin{equation}\label{eq:u(t)=...}
   u(t) = \Phi_\w^t x
      + \dint_0^t \Phi_{\theta^{s}\w}^{t-s} f_{\theta^s \w}(u(s))\,ds,
\end{equation}
which implies that $\Psi$ is a Bochner measurable RDS on $X$ over $\Theta$ (see~\cite[Proposition 2.1]{Barreira_Valls-SD-2018} and also~\cite[Theorem 2.2.1]{Arnold-RDS-1998}).

We also define
\begin{equation}\label{def:sigma}
   \sgm
   = \sup\limits_{(t,\w) \in \R_0^+ \x \W} \
      \dfrac{1}{\alpha^+_{t,\w}}
      \dint_{0}^{t} \alpha^+_{t-s,\theta^{s}\w} \Lip(f_{\theta^s\w}) \alpha^+_{s,\w}\ds
\end{equation}
and
\begin{equation}\label{def:tau}
   \tau
   = \sup_{\w\in\W} \dint_{0}^{+\infty} \
      \alpha^-_{s,\theta^s\w} \Lip(f_{\theta^s\w}) \alpha^+_{s,\w}\ds
\end{equation}
and given $\phi \in \cL_N$ and $\w \in \W$ we denote the \emph{graph} of $\phi_\w$ by
\begin{equation*}
   \cV_{\phi,\w}
   = \set{(\xi,\phi_\w(\xi)): \xi \in E_\w}.
\end{equation*}

\begin{theorem} \label{thm:global}
   Let $\Theta \equiv (\W,\Sgm,\mu,\theta)$ be a measure-preserving dynamical system and let $X$ be a Banach space. Consider a Bochner measurable linear RDS $\Phi$ on $X$ over $\Theta$ that admits a generalized dichotomy with bounds $\alpha^+$ and $\alpha^-$ and let $f \in \sF^{(B)}_\alpha$. Assume that $\Psi$ is a Bochner measurable RDS such that~\eqref{eq:u(t)=...} has a unique solution $\Psi(\cdot,\w,x)$ for every $\w \in \W$ and every $x \in X$. If
   \begin{equation} \label{eq:CondicaoTeo}
      \lim_{t \to +\infty} \alpha^+_{t,\w} \alpha^-_{t,\theta^t\w} = 0\ \ \ \text{ for all } \w\in\W
   \end{equation}
   and
   \begin{equation}\label{eq:CondicaoTeo:alpha+tau<1/2}
      \sgm +  \tau < \frac12,
   \end{equation}
   then there exist $N \in\ ]0,1[$ and a unique $\phi \in \cL^{(B)}_N$ such that
   \begin{equation} \label{thm:global:invar}
      \Psi_{\w}^t(\cV_{\phi,\w})
      \subseteq \cV_{\phi,\theta^t\w}\ \ \ \text{ for all } (t,\w) \in \R_0^+ \x \W.
     \end{equation}
   Furthermore, there is $C \in\ ]0,4[$ depending on $\sgm$ and $\tau$ such that
   \begin{equation}\label{thm:ineq:norm:F_mn(xi...)-F_mn(barxi...)}
      \|\Psi_{\w}^t(\xi,\phi_\w(\xi)) - \Psi_{\w}^t(\bxi,\phi_\w(\bxi))\|
      \le C \alpha^+_{t,\w} \, \|\xi - \bxi\|
   \end{equation}
   for  every $(t,\w, \xi), (t,\w, \bxi) \in \cH$.
\end{theorem}

The proof of Theorem~\ref{thm:global} will be given in Subsection~\ref{Proof of Theorem continuous}.

\subsection{Corollaries}\label{s:corollaries:cont}

In this subsection we are going to state some corollaries of Theorem~\ref{thm:global} that include the tempered exponential dichotomies, as well results covering the different situations given in Example~\ref{ex:2}.

Throughout this subsection we consider a real number $\delta \in\, ]0,1/4[$ and a random variable $G \colon \W \to \, ]0,+\infty[$ such that
\begin{equation*}
   \int_{-\infty}^{+\infty} G(\theta^s\w)\,ds \le 1
   \ \ \ \text{ for all } \w \in \W.
\end{equation*}

\begin{corollary}[{Tempered exponential dichotomies}]\label{corollary:tempered:cont}
   Let $\Theta \equiv (\W,\Sgm,\mu,\theta)$ be a metric dynamical system and let $X$ be a Banach space. Consider a Bochner measurable linear RDS $\Phi$ on $X$ over $\Theta$ that admits a tempered exponential dichotomy with bounds
   \begin{equation*}
       \alpha^+_{t,\w} = K(\w) \e^{a(\w)t}
       \ \ \ \text{ and  } \ \ \
       \alpha^-_{t,\theta^t\w} = K(\theta^t\w) \e^{b(\w)t},
    \end{equation*}
    such that $a(\w)+b(\w) < 0$
    and let $f \in \sF^{(B)}_\alpha$. Assume that $\Psi$ is a Bochner measurable RDS such that~\eqref{eq:u(t)=...} has a unique solution $\Psi(\cdot,\w,x)$ for every $\w \in \W$ and every $x \in X$. Consider a $\theta$-invariant random variable $\gamma(\w)>0$ satisfying $a(\w) + b(\w)+\gamma(\w)<0$.
    If
   \begin{equation*}
      \Lip(f_\w)
      \le \frac{\delta}{K(\w)}\mins{G(\w),
      \frac{\abs{a(\w) + b(\w)+\gamma(\w)}}{\lbd_{K,\gamma(\w),\w}}}\ \ \ \text{ for all $\w \in \W$, }
   \end{equation*}
   then the same conclusions of Theorem~\ref{thm:global} hold.
\end{corollary}

\begin{proof}
   Since $K$ is tempered we have
   \begin{equation*}
       \lim_{t \to +\infty} \alpha^+_{t,\w} \alpha^-_{t,\theta^t\w}=\lim_{t \to +\infty}
      K(\w)K(\theta^t \w) \e^{(a(\w) + b(\w))t}
      = 0
         \ \ \ \text{ for all } \w \in \W
   \end{equation*}
and~\eqref{eq:CondicaoTeo} holds.
   From
   \begin{align*}
      \dfrac{1}{\alpha^+_{t,\w}} \int_{0}^{t}
         \alpha^+_{t-s,\theta^{s}\w} \Lip(f_{\theta^s\w}) \alpha^+_{s,\w}\ds
      & = \int_{0}^{t}
         K(\theta^{s} \w) \Lip(f_{\theta^s \w})\ds\\
      & \le \delta \int_{-\infty}^{+\infty} G(\theta^s\w)\ds\\
      & \le \delta,
   \end{align*}
   we conclude that $\sgm \le \delta$. On the other hand, since $K(\w)\leq \e^{\gamma(\w) \abs{s}} \lbd_{K,\gamma(\w),\theta^s\w}$ for every $\w\in\W$ and $s\in\R$, we have
   \begin{align*}
      & \dint_{0}^{+\infty}
         \alpha^-_{s,\theta^s\w} \Lip(f_{\theta^s\w}) \alpha^+_{s,\w} \ds\\
      & = \dint_{0}^{+\infty} K(\w)
         K(\theta^s\w)\e^{a(\w)+b(\w)} \Lip(f_{\theta^s\w})\,ds\\
      & \le \delta \dint_{0}^{+\infty}  -(a(\w)+b(\w)+\gamma(\w))\e^{(a(\w)+b(\w)+\gamma(\w))s} \,ds\\
      & \le \delta.
   \end{align*}
    This implies $\sigma+\tau \le 2\delta < 1/2$ and consequently we are in conditions to apply Theorem~\ref{thm:global}.
\end{proof}

We consider now dichotomies with bounds of the form
    \begin{equation}\label{eq:integral_tempered}
      \alpha^+_{t,\w} = K(\w) \e^{\int_0^t a(\theta^r\w)\,dr}
      \ \ \ \text{ and  } \ \ \
      \alpha^-_{t,\theta^t\w} = K(\theta^t \w) \e^{\int_0^t b(\theta^r\w)\,dr},
   \end{equation}
   where $K\colon\W\to[1,+\infty[$ is a random variable such that for all $\w\in\W$ and all $t\in\R$ the following derivative exists:
   \[
   d_\w(t)=\frac{d}{dt}[K(\theta^t\w)]=\lim_{h\to0}\dfrac{K(\theta^{t+h}\w)-K(\theta^{t} \w)}{h}.
   \]
We notice that for all $s\in\R$, $d_{\theta^s\w}(0)=d_\w(s)$.

\begin{corollary}\label{corollary:tempered:integral:cont}
   Let $\Theta \equiv (\W,\Sgm,\mu,\theta)$ be a measure-preserving dynamical system and let $X$ be a Banach space. Consider a Bochner measurable linear RDS $\Phi$ on $X$ over $\Theta$ that admits a dichotomy with bounds given by~\eqref{eq:integral_tempered} satisfying $d_\w(0)>K(\w)(a(\w)+b(\w))$ for all $\w \in \W$. Let $f \in \sF^{(B)}_\alpha$ be such that
   \begin{equation*}
      \Lip(f_\w)
      \le \dfrac{\delta}{K(\w)}
         \mins{G(\w), \dfrac{1}{K(\w)}
            \prts{\dfrac{d_\w(0)}{K(\w)} -(a(\w)+b(\w))}}
   \end{equation*}
   for all $\w \in \W$. Assume that $\Psi$ is a Bochner measurable RDS such that~\eqref{eq:u(t)=...} has a unique solution $\Psi(\cdot,\w,x)$ for every $\w \in \W$ and every $x \in X$. If
   \begin{equation}\label{eq:corol:exp_dich}
      \dlim_{t \to +\infty}
      K(\theta^t \w) \e^{\int_0^t a(\theta^r \w) + b(\theta^r \w)\dr}
      = 0
         \ \ \ \text{ for all } \w \in \W,
   \end{equation}
   then the same conclusions of Theorem~\ref{thm:global} hold.
\end{corollary}

\begin{proof}
   It is obvious that~\eqref{eq:corol:exp_dich} is equivalent to~\eqref{eq:CondicaoTeo} and as in the proof of Corollary~\ref{corollary:tempered:cont} we get $\sgm \le \delta$. On the other hand, since
   \begin{align*}
      &
         \frac{d}{dt}
         \prts{\dfrac{\e^{\int_0^t a(\theta^r \w) + b(\theta^r \w)\dr}}
         {K(\theta^{t}\w)}}\\
      &
         = - \dfrac{\e^{\int_0^t a(\theta^r \w) + b(\theta^r \w)\dr}}
            {K(\theta^t\w)}
            \prts{\dfrac{d_{\w}(t)}{K(\theta^t\w)}-(a(\theta^t \w)+b(\theta^t \w))},
   \end{align*}
   we have
   \begin{align*}
      & \dint_{0}^{+\infty}
         \alpha^-_{s,\theta^s\w} \Lip(f_{\theta^s\w}) \alpha^+_{s,\w} \ds\\
      & = K(\w) \dint_{0}^{+\infty}
         K(\theta^s\w)\e^{\int_0^s a(\theta^r \w)+b(\theta^r \w)) \dr}  \Lip(f_{\theta^s\w})\,ds\\
      & \le \delta K(\w)\dint_{0}^{+\infty}
         \dfrac{\e^{\int_0^s a(\theta^r \w)+b(\theta^r \w)) \dr}}{K(\theta^s\w)}\prts{\dfrac{d_{\theta^s\w}(0)}{K(\theta^s\w)}-(a(\w)+b(\w))}\,ds\\
      & = \delta
         - \delta\lim_{s \to +\infty}
         K(\w)\dfrac{\e^{\int_0^s a(\theta^r \w)+b(\theta^r \w)) \dr}}{K(\theta^{s}\w)}\\
      & \le \delta
         - \delta \lim_{s \to +\infty} K(\w)
         \e^{\int_0^s a(\theta^r \w)+b(\theta^r \w)) \dr}\\
      & = \delta.
   \end{align*}
    Thus $\sigma+\tau \le 2\delta < 1/2$ and we are in the conditions of Theorem~\ref{thm:global}.
\end{proof}

Given random variables $a, b \colon \W \to \R_0^+$ and $K \colon \W \to [1,+\infty[$, we consider now dichotomies with bounds of the form
\begin{equation}\label{dic:a,b,c}
   \alpha^+_{t,\w}
   = K(\w) \dfrac{a(\w)}{a(\theta^t\w)}
   \ \ \ \text{ and } \ \ \
   \alpha^-_{t,\theta^t\w}
   = K(\theta^t\w) \dfrac{b(\w)}{b(\theta^t\w)}
\end{equation}
and define the function $H_\w\colon\R\to\R$ by
\begin{equation*}
   H_\w(s)=-\dfrac{1}{a(\theta^s\w)b(\theta^s\w)K(\theta^s\w)}.
\end{equation*}

\begin{corollary}\label{corollary:a:b:cont}
   Let $\Theta \equiv (\W,\Sgm,\mu,\theta)$ be a measure-preserving dynamical system and let $X$ be a Banach space. Consider a Bochner measurable linear RDS $\Phi$ on $X$ over $\Theta$ that admits a dichotomy with bounds given by~\eqref{dic:a,b,c} and such that the map $H_\w$  is differentiable (except, eventually, on a finite set of points) and  $H_\w'(s)>0$ for all $s\in\R$. Let $f \in \sF^{(B)}_\alpha$ satisfying
   \begin{equation*}
      \Lip(f_\w)
      \le \frac{\delta}{K(\w)} \mins{G(\w),a(\w)b(\w)H_\w'(0)}
   \end{equation*}
   for all $\w \in \W$. Assume that $\Psi$ is a Bochner measurable RDS such that~\eqref{eq:u(t)=...} has a unique solution $\Psi(\cdot,\w,x)$ for every $\w \in \W$ and every $x \in X$. If
   \begin{equation}\label{eq:example:lim d(t)/a(t)b(t)=0:cont}
      \lim_ {t\to+\infty}\frac{K(\theta^t\w)}{a(\theta^t\w)b(\theta^t\w)}=0
      \text{ for all } \w \in \W,
   \end{equation}
   then the same conclusions of Theorem~\ref{thm:global} hold.
\end{corollary}

\begin{proof}
   From~\eqref{eq:example:lim d(t)/a(t)b(t)=0:cont} we have that~\eqref{eq:CondicaoTeo} holds. Let us check now~\eqref{eq:CondicaoTeo:alpha+tau<1/2}. Once again, as in the proof of Corollary~\ref{corollary:tempered:cont} we get
   $\sigma\le\delta$ and all $\w \in \W$. To estimate $\tau$ notice that $H_{\theta^r\w}(s)=H_\w(r+s)$ and $H_{\theta^r\w}'(s)=H_\w'(r+s)$ for all $r, s\in\R$ and all $\w \in \W$. In view of this we have
   \begin{align*}
      \int_0^{+\infty}
         \alpha^-_{s,\theta^s\w} \Lip(f_{\theta^s\w}) \alpha^+_{s,\w}\ds
      & \leq \delta a(\w)b(\w)K(\w) \int_0^{+\infty}  H_\w'(s) \ds\\
      &= \delta a(\w)b(\w)K(\w) \prtsr{\lim_{s\to +\infty}H_\w(s)-H_\w(0)}\\
      &= \delta + \delta a(\w)b(\w)K(\w)\lim_{s\to+\infty}H_\w(s).
   \end{align*}
   Since $K(\w)\ge1$ for all $\w\in\W$, it follows from~\eqref{eq:example:lim d(t)/a(t)b(t)=0:cont} that
   \begin{align*}
      \lim_{s\to+\infty}H_\w(s)
      & =\lim_{s\to+\infty}
         -\frac{1}{a(\theta^{s}\w) b(\theta^{s}\w) K(\theta^{s}\w)}\\
      & = \lim_{s\to+\infty}
         -\frac{K(\theta^s\w)}{a(\theta^s\w)b(\theta^s\w)} \frac{1}{K(\theta^s\w)^2}\\
      & = 0.
   \end{align*}
   Hence
   \[
      \int_0^{+\infty}
         \alpha^-_{s,\theta^s\w} \Lip(f_{\theta^s\w}) \alpha^+_{s,\w} \ds
     \leq \delta,
   \]
   which implies $\tau\le\delta$, and consequently $\sigma+\tau\le2\delta<1/2$.
\end{proof}

\subsection{Proof of Theorem~\ref{thm:global}}\label{Proof of Theorem continuous}

In this section we prove Theorem~\ref{thm:global}. We start by fixing suitable constants $M$ and $N$ to be used in this proof.

\begin{lemma}[{\cite[Lemma 5.1]{Bento_Costa-EJQTDE-2017}}] \label{lemma:M&N}
   If $\sgm$ and $\tau$ are positive real numbers such that $\sgm + \tau <1/2$,
   then there exist $M \in\ ]1,2[$ and $N \in\ ]0,1[$ such that
   \begin{equation}\label{def:M+N}
      \sgm = \dfrac{M-1}{M(1+N)}
      \ \ \ \text{ and } \ \ \
      \tau = \dfrac{N}{M(1+N)}.
   \end{equation}
\end{lemma}

To prove Theorem~\ref{thm:global} we will use Banach Fixed Point Theorem to find a convenient $\phi\in\cL^{(B)}_N$.

\begin{lemma}\label{lemma:bochner-measurability}
   Let $(h,\phi) \in \cX^{(B)}_{M,N}$. Then the maps
   \begin{equation}\label{eq:(t,s,w)|->Phi_P_f...:Bochner:meas}
      \begin{split}
         & (t,s,\w) \mapsto \Phi^{t-s}_{\theta^s \w} P_{\theta^s \w}
            f_{\theta^s \w}(h_{s,\w}(P_\w x),\phi_{\theta^s \w}(h_{s,\w}(P_\w x)))\\
         & (s,\w) \mapsto \Phi^{-s}_{\theta^s \w} Q_{\theta^s \w}
            f_{\theta^s \w}(h_{s,\w}(P_\w x),\phi_{\theta^s \w}(h_{s,\w}(P_\w x)))
      \end{split}
   \end{equation}
   are Bochner measurable for every $x \in X$ and the maps
   \begin{equation}\label{eq:s|->Phi_P_f...:Bochner:int}
      \begin{split}
         & s \mapsto \Phi^{t-s}_{\theta^s \w} P_{\theta^s \w}
            f_{\theta^s \w}(h_{s,\w}(P_\w x),\phi_{\theta^s \w}(h_{s,\w}(P_\w x)))\\
         & s \mapsto \Phi^{-s}_{\theta^s \w} Q_{\theta^s \w}
            f_{\theta^s \w}(h_{s,\w}(P_\w x),\phi_{\theta^s \w}(h_{s,\w}(P_\w x)))
      \end{split}
   \end{equation}
   are Bochner integrable in $[0,t]$ for every $(t,\w,x) \in \T^+ \x \W \x X$.
\end{lemma}

\begin{proof}
   Let $(h,\phi) \in \cX^{(B)}_{M,N}$. First we prove that the maps~\eqref{eq:(t,s,w)|->Phi_P_f...:Bochner:meas} are Bochner measurable. Since $\theta$ is $(\cB(\T^+) \otimes \Sgm,\Sgm)$-measurable, from~\eqref{eq:w|->phi_w(xi):Bochner}, \eqref{eq:w|->f_w(x):Bochner} and Corollary~\ref{cor:Bmeas_o_meas_is_Bmeas}, it follows that
   \begin{align*}
      & (s,\w) \mapsto \phi_{\theta^s \w}(P_\w x)
         \text{ is Bochner measurable for every } x \in X;\\
      & (s,\w) \mapsto f_{\theta^s \w} (x)
         \text{ is Bochner measurable for every } x \in X.
   \end{align*}
   Analogous, since $\Phi$ is a Bochner measurable linear RDS,
   \begin{equation*}
      (t,s,\w) \mapsto \Phi^{t-s}_{\theta^s\w} x \text{ is Bochner measurable for every } x \in X
   \end{equation*}
   and by~\cite[Corollary 1.1.29]{Hytonen_Neerven_Veraar_Weis-book-2016} we also have
   \begin{equation*}
      (t,s,\w) \mapsto \Phi^{t-s}_{\theta^s\w} P_{\theta^s \w} x \text{ is Bochner measurable for every } x \in X.
   \end{equation*}
   Hence, from~\eqref{eq:(t,w)|->h_t,w(xi):Bochner} and~\cite[Lemma 2.2]{Aulbach_Wanner-book-1996} it follows that
   \begin{equation*}
      (s,\w) \mapsto \phi_{\theta^s \w}(h_{s,\w}(P_\w x))
      \ \ \ \text{ and } \ \ \
      (s,\w) \mapsto
         f_{\theta^s \w}(h_{s,\w}(P_\w x),\phi_{\theta^s \w}(h_{s,\w}(P_\w x)))
   \end{equation*}
   are Bochner measurable for all $x \in X$. By~\cite[Proposition 1.1.28]{Hytonen_Neerven_Veraar_Weis-book-2016} the maps~\eqref{eq:(t,s,w)|->Phi_P_f...:Bochner:meas} are Bochner measurable for every $x \in X$.

   From~\eqref{ine:||f_w(x)||<=Lip(f_w)||x||},~\eqref{ine:||phi_w(xi)||<=N||xi||} and~\eqref{ine:||h_n,w(xi)||<=M...} it follows that
   \begin{align*}
      & \norm{f_{\theta^s \w}(h_{s,\w}(P_\w x), \phi_{\theta^s \w}(h_{s,\w}(P_\w x)))}\\
      & \le \Lip(f_{\theta^s \w}) \prts{\norm{h_{s,\w}(P_\w x)}
         + \norm{\phi_{\theta^s \w}(h_{s,\w}(P_\w x))}}\\
      & \le \Lip(f_{\theta^s \w}) \prts{\norm{h_{s,\w}(P_\w x)}
         + N \norm{(h_{s,\w}(P_\w x))}}\\
      & \le M(1+N) \Lip(f_{\theta^s \w}) \alpha^+_{s,\w} \|P_\w x\|.
   \end{align*}
   By~\ref{eq:split4} we have
   \begin{align*}
      & \norm{\Phi^{t-s}_{\theta^s \w} P_{\theta^s \w}
         f_{\theta^s \w}(h_{s,\w}(P_\w x),\phi_{\theta^s \w}(h_{s,\w}(P_\w x)))}\\
      & \le M(1+N)
         \alpha^+_{t-s,\theta^s \w} \Lip(f_{\theta^s \w}) \alpha^+_{s,\w} \|P_\w x\|
   \end{align*}
   and by~\ref{eq:split5} we obtain
   \begin{equation}\label{ine:||Phi^-s...f...||<=M(1+N)...}
      \begin{split}
         & \norm{\Phi^{-s}_{\theta^s \w} Q_{\theta^s \w}
            f_{\theta^s \w}(h_{s,\w}(P_\w x),\phi_{\theta^s \w}(h_{s,\w}(P_\w x)))}\\
      & \le M(1+N)
         \alpha^-_{s,\theta^s\w} \Lip(f_{\theta^s \w}) \alpha^+_{s,\w} \|P_\w x\|.
      \end{split}
   \end{equation}
   Thus, taking into account the definitions of $\sigma$ and $\tau$, we have
   \begin{equation*}
      \dint_0^t \alpha^+_{t-s,\theta^s \w} \Lip(f_{\theta^s \w}) \alpha^+_{s,\w} \ds
      \le \sigma \alpha^+_{t,\w}
      < +\infty
   \end{equation*}
   and
   \begin{equation*}
      \dint_0^t \alpha^-_{s,\theta^s\w} \Lip(f_{\theta^s \w}) \alpha^+_{s,\w}
      \le \dint_0^{+\infty}
         \alpha^-_{s,\theta^s\w} \Lip(f_{\theta^s \w}) \alpha^+_{s,\w}
      \le \tau
      < +\infty,
   \end{equation*}
   which by Proposition~\ref{prop:Bochner:2} imply that the maps~\eqref{eq:s|->Phi_P_f...:Bochner:int} are Bochner integrable in $[0,t]$ for every $(t,\w,x) \in \T^+ \x \W \x X$.
\end{proof}

Given $\w\in \W$ and $v_\w=(\xi_\w,\eta_\w)\in E_\w \times F_\w$,
it follows from~\eqref{eq:Psi:cont} that the trajectory
$v_{\theta^t\w} = \Psi^t_\w v_\w$, with $v_{\theta^t\w}=\prts{x_{\theta^t\w},y_{\theta^t\w}}\in E_{\theta^t\w}\x F_{\theta^t\w}$ satisfies the following equations
\begin{align}
   x_{\theta^t\w} & = \Phi_{\w}^t \xi_\w + \int_0^t \Phi_{\theta^s\w}^{t-s} P_{\theta^s\w} f_{\theta^s\w}(x_{\theta^s\w},y_{\theta^s\w})\ds,
      \label{eq:dyn-split1a}\\
   y_{\theta^t\w}  &= \Phi_{\w}^t \eta_\w +\int_0^t \Phi_{\theta^s\w}^{t-s} Q_{\theta^{s}\w} f_{\theta^s\w}(x_{\theta^s\w},y_{\theta^s\w})\ds
      \label{eq:dyn-split1b}
\end{align}
for each $t \in\R_0^+$. In view of the forward invariance required in \eqref{thm:global:invar}, each
trajectory given by~\eqref{eq:Psi:cont} starting in $\cV_{\phi,\w}$ must be in
$\cV_{\phi,\theta^t\w}$ for every $t \in \R_0^+$, and thus equations~\eqref{eq:dyn-split1a} and~\eqref{eq:dyn-split1b} can be written in
the form
\begin{align}
   x_{\theta^t\w} & = \Phi_{\w}^t \xi_\w + \int_0^t \Phi_{\theta^s\w}^{t-s} P_{\theta^s\w} f_{\theta^s\w}(x_{\theta^s\w},\phi_{\theta^s\w}(x_{\theta^s\w}))\ds,
      \label{eq:dyn-split2a}\\
   \phi_{\theta^t\w}(x_{\theta^t\w})  &= \Phi_{\w}^t  \phi_{\w}(\xi_\w) +\int_0^t \Phi_{\theta^s\w}^{t-s} Q_{\theta^{s}\w} f_{\theta^s\w}(x_{\theta^s\w},\phi_{\theta^s\w}(x_{\theta^s\w}))\ds.
      \label{eq:dyn-split2b}
\end{align}

In the following we rewrite conditions~\eqref{eq:dyn-split2a}
and~\eqref{eq:dyn-split2b}.
\begin{lemma} \label{lemma:equiv}
   Consider $(h,\phi) \in \cX^{(B)}_{M,N}$ such that
   \begin{equation} \label{eq:dyn-split3a}
      h_{t,\w}(\xi)
      = \Phi_{\w}^t \xi + \int_0^t \Phi_{\theta^s\w}^{t-s} P_{\theta^s\w}
         f_{\theta^s\w}(h_{s,\w}(\xi),\phi_{\theta^s\w}(h_{s,\w}(\xi)))\ds
   \end{equation}
   for all $(t,\w,\xi)\in \cH$. Then the following properties are equivalent:
   \begin{enumerate}[\lb=$\alph*)$,\lm=5mm]
      \item for every $(t,\w,\xi)\in \cH$
        \begin{align}
            \phi_{\theta^t\w}(h_{t,\w}(\xi))  &= \Phi_{\w}^t  \phi_{\w}(\xi) +\int_0^t \Phi_{\theta^s\w}^{t-s} Q_{\theta^s\w} f_{\theta^s\w}(h_{s,\w}(\xi),\phi_{\theta^s\w}(h_{s,\w}(\xi)))\ds;
            \label{eq:dyn-split3b}
        \end{align}
      \item for every $(\w,\xi) \in \cG$
        \begin{equation} \label{eq:phi_n}
            \phi_{\w}(\xi)
               = -\int_0^{+\infty} \phiw{s} Q_{\theta^s\w} f_{\theta^s\w}(h_{s,\w}(\xi),\phi_{\theta^s\w}(h_{s,\w}(\xi)))\ds.
        \end{equation}
   \end{enumerate}
\end{lemma}

\begin{proof}
   First we prove that the integral in equation~\eqref{eq:phi_n} is convergent.
   Indeed, from~\eqref{ine:||Phi^-s...f...||<=M(1+N)...} and~\eqref{def:tau}, we conclude that for every $(\w,\xi) \in \cG$
   \begin{align*}
      & \int_0^{+\infty} \|\phiw{s} Q_{\theta^s\w}
        f_{\theta^s\w}(h_{s,\w}(\xi),\phi_{\theta^s\w}(h_{s,\w}(\xi)))\|\ds\\
      & \le M(1+N) \|\xi\|
         \int_0^{+\infty} \alpha^-_{s,\theta^s\w} \Lip(f_{\theta^s \w})\alpha^+_{s,\w}\ds\\
      & \le M(1+N)\tau \|\xi\|.
   \end{align*}

   Now, let us suppose that~\eqref{eq:dyn-split3b} holds for every $(t,\w,\xi) \in \cH$. Then, since for $t > s$ we have
      $$ \phiw{t}\prts{\Phi_{\theta^s\w}^{t-s}\vert_{F_{\theta^s\w}}}
      =\phiw{s},$$
   equation~\eqref{eq:dyn-split3b} can be written in the following equivalent form (apply $\Phi_{\theta^t \w}^{-t}$ to both sides)
   \begin{equation} \label{eq:equiv1A}
    \phi_{\w}(\xi)= \phiw{t}\phi_{\theta^t\w}(h_{t,\w}(\xi))  -\int_0^t \phiw{s} Q_{\theta^s\w} f_{\theta^s\w}(h_{s,\w}(\xi),\phi_{\theta^s\w}(h_{s,\w}(\xi)))\ds.
   \end{equation}\normalsize
   Using~\ref{eq:split5},~\eqref{ine:||phi_w(xi)||<=N||xi||} and~\eqref{ine:||h_n,w(xi)||<=M...},
   we have
   \begin{align*}
      \|\phiw{t}\phi_{\theta^t\w}(h_{t,\w}(\xi))\|
      & = \|\phiw{t}Q_{\theta^t\w}\phi_{\theta^t\w}(h_{t,\w}(\xi))\| \\
      & \le N\|h_{t,\w}(\xi)\|\alpha^-_{t,\theta^t\w}\\
      & \le MN\|\xi\|\alpha^+_{t,\w}\alpha^-_{t,\theta^t\w}
   \end{align*}
   and by~\eqref{eq:CondicaoTeo} this converge to zero when $t\to +\infty$.
   Hence, letting $t\to +\infty$ in \eqref{eq:equiv1A} we obtain the
   identity~\eqref{eq:phi_n} for every $(\w,\xi) \in \cG$.

   Assume now that~\eqref{eq:phi_n} holds for every $(\w,\xi) \in \cG$. Applying $\Phi_\w^t$ to both sides, we have for all $t \in\R_0^+$
   \begin{align*}
      \Phi_\w^t\phi_\w(\xi)
               &= -\int_0^{+\infty} \Phi_\w^t\phiw{s} Q_{\theta^s\w} f_{\theta^s\w}(h_{s,\w}(\xi),\phi_{\theta^s\w}(h_{s,\w}(\xi)))\ds\\
               &=-\int_0^t \Phi_{\theta^s\w}^{t-s} Q_{\theta^s\w} f_{\theta^s\w}(h_{s,\w}(\xi),\phi_{\theta^s\w}(h_{s,\w}(\xi)))\ds\\
                  &\,\,\,\, -\int_{t}^{+\infty} \phii{s}{(s-t)} Q_{\theta^s\w} f_{\theta^s\w}(h_{s,\w}(\xi),\phi_{\theta^s\w}(h_{s,\w}(\xi)))\ds\\
               &=-\int_0^t \Phi_{\theta^s\w}^{t-s} Q_{\theta^s\w} f_{\theta^s\w}(h_{s,\w}(\xi),\phi_{\theta^s\w}(h_{s,\w}(\xi)))\ds\\
                  &\,\,\,\, -\int_0^{+\infty} \phii{t+s}{s} Q_{\theta^{t+s}\w} f_{\theta^{t+s}\w}(h_{t+s,\w}(\xi),\phi_{\theta^{t+s}\w}(h_{t+s,\w}(\xi)))\ds,
      \end{align*}
   and thus~\eqref{eq:dyn-split3b} holds due to the uniqueness of the solution of~\eqref{eq:Psi:cont}, which in particular implies that $h_{t+s,\w}(\xi)$ can be replaced by $h_{s,\theta^t\w}(h_{t,\w}(\xi))$.
\end{proof}

Let $J$ be the operator that assigns to every $(h,\phi)\in\cX^{(B)}_{M,N}$ the function $J(h,\phi)\colon \cH \to X$ defined by
\begin{align*}
   \prtsr{J\prts{h,\phi}}(t,\w,\xi)
   & = \Phi_\w^t \xi + \dint_0^t
      \Phi_{\theta^s\w}^{t-s} P_{\theta^s\w}
      f_{\theta^s\w}(h_{s,\w}(\xi),\phi_{\theta^s\w}(h_{s,\w}(\xi))) \ds.
\end{align*}
Notice that by Lemma~\ref{lemma:bochner-measurability} the operator $J$ is well defined.

\begin{lemma}
 $J\prts{\cX^{(B)}_{M,N}}\subseteq\cJ^{(B)}_M.$
\end{lemma}

\begin{proof}
   Let $(h,\phi) \in \cX^{(B)}_{M,N}$. Denoting $[J\prts{h,\phi}](t,\w,\xi)$ by $J\prts{h,\phi}_{t,\w}(\xi)$, it follows from Lemma~\ref{lemma:bochner-measurability} that
   \begin{equation*}
      (t,\w) \mapsto J\prts{h,\phi}_{t,w}(P_\w x)
      \text{ is Bochner measurable for every } x \in X.
   \end{equation*}

   By definition, for  every $(h,\phi) \in
   \cX^{(B)}_{M,N}$ we have $J\prts{h,\phi}_{t,\w}(\xi) \in E_{\theta^t\w}$ for every $(t,\w, \xi) \in \cH$ and $J\prts{h,\phi}_{0,\w}(\xi) = \xi$ for every $(\w,\xi) \in \cG$.  Moreover, from~\eqref{eq:h_n,w(0)=0},~\eqref{eq:phi_w(0)=0} and~\eqref{eq:f_w(0)=0}, it follows that $J(h,\phi)_{t,\w}(0)=0$ for every $(t,\w) \in \R_0^+ \x \W$.
   Hence $J(h,\phi)$ satisfies~\eqref{eq:(t,w)|->h_t,w(xi):Bochner},~\eqref{eq:h_n,w(0)=0},
   \eqref{eq:h_0,w(xi)=xi} and~\eqref{h_n,w(xi):in:E(t^n(w))}.

   Let us see that $J\prts{h,\phi}$ satisfies
   \eqref{ine:||h_n,w(xi)-h_n,w(bxi)||<=M...}. Defining
      $$ \gamma_{\theta^s\w}(\xi,\bxi)
         = \| f_{\theta^s\w}(h_{s,\w}(\xi),\phi_{\theta^s\w}(h_{s,\w}(\xi)))
            - f_{\theta^s\w}(h_{s,\w}(\bxi),\phi_{\theta^s\w}(h_{s,\w}(\bxi)))\|,$$
   by~\ref{eq:split4} we have
   \begin{equation}\label{ine:||(J_m,n(x,phi))(xi)-(J_m,n(x,phi))(barxi)||}
      \begin{split}
          & \| J\prts{h,\phi}_{t,\w}(\xi)
            - J\prts{h,\phi}_{t,\w}(\bar\xi)\|\\
         & \le \|\Phi_\w^tP_{\w}\| \|\xi - \bar\xi\|
            + \dint_0^t \|\Phi_{\theta^s\w}^{t-s} P_{\theta^s\w}\| \gamma_{\theta^s\w}(\xi,\bxi)\ds\\
         & \le  \|\xi - \bar\xi\|\alpha^+_{t,\w}
            + \dint_0^t \alpha^+_{t-s,{\theta^s\w}} \gamma_{\theta^s\w}(\xi,\bxi)\ds,
      \end{split}
   \end{equation}

   From~\eqref{ine:||f_w(x)-f_w(y)||<=Lip(f_w)||x-y||},~\eqref{eq:||phi_w(xi)-pi_w(bxi)||<=...} and~\eqref{ine:||h_n,w(xi)-h_n,w(bxi)||<=M...} we have
   \begin{equation}\label{eq:||f_k(x_k,n(xi),...)-f_k(x_k,n(barxi),...)||}
      \begin{split}
         & \gamma_{\theta^s\w}(\xi,\bxi)\\
         & \le \Lip(f_{\theta^s\w}) \prts{\|h_{s,\w}(\xi) - h_{s,\w}(\bar\xi)\|
            + \|\phi_{\theta^s\w}\prts{h_{s,\w}(\xi)} - \phi_{\theta^s\w}\prts{h_{s,\w}(\bar\xi)}\|}\\
         & \le \Lip(f_{\theta^s\w}) \prts{\|h_{s,\w}(\xi) - h_{s,\w}(\bar\xi)\|
            + N \|h_{s,\w}(\xi) - h_{s,\w}(\bar\xi)\|}\\
         & \le \Lip(f_{\theta^s\w}) M(1+N) \|\xi-\bar\xi\| \alpha^+_{s,\w}
      \end{split}
   \end{equation}
   which, together with~\eqref{ine:||(J_m,n(x,phi))(xi)-(J_m,n(x,phi))(barxi)||},~\eqref{def:sigma} and~\eqref{def:M+N}, implies
   \begin{equation*}
      \begin{split}
         & \| J\prts{h,\phi}_{t,\w}(\xi)
            - J\prts{h,\phi}_{t,\w}(\bxi)\|\\
         & \le \|\xi - \bxi\|\alpha^+_{t,\w}
            +  M(1+N) \|\xi - \bxi\| \dint_0^t
            \alpha^+_{t-s,\theta^s\w} \Lip(f_{\theta^s\w}) \alpha^+_{s,\w} \ds\\
         & \le \|\xi - \bar\xi\|\alpha^+_{t,\w} + \sgm  M(1+N) \|\xi-\bar\xi\|\alpha^+_{t,\w}\\
         & = \prts{1 + \sgm M(1+N) }  \|\xi - \bar\xi\|\alpha^+_{t,\w}\\
         & = M  \|\xi - \bar\xi\|\alpha^+_{t,\w}.
      \end{split}
   \end{equation*}
\end{proof}

Let now $L$ be the operator that assigns to every $(h,\phi)\in\cX^{(B)}_{M,N}$ the function $L(h,\phi)\colon \cG \to X$ defined by
\begin{align*}
   \prtsr{L\prts{h,\phi}}(\w,\xi)
   & = -\int_0^{+\infty} \phiw{s} Q_{\theta^s\w} f_{\theta^s\w}(h_{s,\w}(\xi),\phi_{\theta^s\w}(h_{s,\w}(\xi))) \ds,
\end{align*}
that by Lemmas~\ref{lemma:bochner-measurability} and~\ref{lemma:equiv} is well defined.

\begin{lemma}
 $L\prts{\cX^{(B)}_{M,N}}\subseteq\cL^{(B)}_N.$
\end{lemma}

\begin{proof}
   As before, we denote  $\prtsr{L\prts{h,\phi}}(\w,\xi)$ by $L\prts{h,\phi}_\w(\xi)$. By~\eqref{eq:(t,s,w)|->Phi_P_f...:Bochner:meas} we can conclude that
   \begin{equation*}
      \w \mapsto L(h,\phi)_\w(P_\w x)
      \text{ is Bochner measurable for every } x \in X.
   \end{equation*}
   Moreover, from~\eqref{eq:h_n,w(0)=0},~\eqref{eq:phi_w(0)=0} and~\eqref{eq:f_w(0)=0} we conclude that $L(h,\phi)_\w(0) = 0$ for every $\w \in \W$ and by definition $L(h,\phi)$ satisfies~\eqref{eq:phi_w(xi):in:F_w}.

   Finally, from~\ref{eq:split5}, ~\eqref{eq:||f_k(x_k,n(xi),...)-f_k(x_k,n(barxi),...)||},~\eqref{def:tau} and~\eqref{def:M+N} it follows for every $\prts{h,\phi}\in\cX^{(B)}_{M,N}$ that
   \begin{equation*}
      \begin{split}
         & \|L\prts{h,\phi}_\w(\xi) - L\prts{h,\phi}_\w(\bar\xi)\|\\
         & \le \dint_0^{+\infty} \|\phiw{s}Q_{\theta^s\w} \|
            \, \gamma_{\theta^s\w}(\xi,\bxi) \ds\\
         & \le \int_0^{+\infty}
            \alpha^-_{s,\theta^s\w} \Lip(f_{\theta^s\w}) M(1+N) \|\xi-\bar\xi\| \alpha^+_{s,\w} \ds\\
         & \le \tau M \prts{1 + N} \|\xi-\bar\xi\|\\
         & = N \|\xi-\bar\xi\|
      \end{split}
   \end{equation*}
   and $L(h,\phi) \in \cL^{(B)}_N$.
\end{proof}

We define now the operator $T \colon \cX^{(B)}_{M,N}  \to \cX^{(B)}_{M,N} $ by
   $$ T(h,\phi) = \prts{J(h,\phi), L(h,\phi)}.$$

\begin{lemma}\label{lemma:contraction}
   The operator $T \colon \cX^{(B)}_{M,N} \to \cX^{(B)}_{M,N}$ is a contraction.
\end{lemma}

\begin{proof}
   Let $\prts{h,\phi}, \prts{g,\psi} \in \cX^{(B)}_{M,N}$. Then, setting
   \begin{equation*}
               \hat\gamma_{\theta^s\w}(\xi)= \|f_{\theta^s\w}(h_{s,\w}(\xi),\phi_{\theta^s\w}(h_{s,\w}(\xi)))
            - f_{\theta^s\w}(g_{s,\w}(\xi),\psi_{\theta^s\w}(g_{s,\w}(\xi)))\|,
   \end{equation*}
   by~\eqref{ine:||f_w(x)-f_w(y)||<=Lip(f_w)||x-y||}, \eqref{eq:||phi_w(xi)-pi_w(bxi)||<=...}, \eqref{def:d_1(x,y)},~\eqref{def:d_2(phi,psi)} and~\eqref{ine:||h_n,w(xi)||<=M...} we have
   \begin{equation}\label{ine:hat_gamma<=...}
      \begin{split}
         & \hat\gamma_{\theta^s\w}(\xi)\\
         & \le \Lip(f_{\theta^s\w}) \prts{\|h_{s,\w}(\xi) - g_{s,\w}(\xi)\|
            + \|\phi_{\theta^s\w}(h_{s,\w}(\xi)) - \psi_{\theta^s\w}(g_{s,\w}(\xi))\|}\\
         & \le \Lip(f_{\theta^s\w}) \prts{(1+N)\|h_{s,\w}(\xi) - g_{s,\w}(\xi)\|
            + \|\phi_{\theta^s\w}(g_{s,\w}(\xi)) - \psi_{\theta^s\w}(g_{s,\w}(\xi))\|}\\
         & \le \Lip(f_{\theta^s\w}) \prts{(1+N)d_1(h,g) \ \|\xi\| \alpha^+_{s,\w}
            + M \ d_2(\phi,\psi) \ \|\xi\| \alpha^+_{s,\w}}\\
         & \le \Lip(f_{\theta^s\w})  \|\xi\|\alpha^+_{s,\w}
            \prts{(1+N)d_1(h,g) + M d_2(\phi,\psi) }
      \end{split}
   \end{equation}
   for every $\w \in \W$, $s \in \R_0^+$ and $\xi \in E_\w$. Thus from~\ref{eq:split4}, last inequality and~\eqref{def:sigma}, it follows that
   \begin{align*}
      & \|J(h,\phi)_{t,\w}(\xi) - J(g,\psi)_{t,\w}(\xi)\|\\
      & \le  \dint_0^t \|\Phi_{\theta^s\w}^{t-s} P_{\theta^s\w}\| \hat\gamma_{\theta^s\w}(\xi) \ds\\
      & \le   \dint_0^{t} \alpha^+_{t-s,\theta^s\w}\Lip(f_{\theta^s\w})  \|\xi\|\alpha^+_{s,\w} \prts{(1+N)d_1(h,g)
            + M d_2(\phi,\psi) }\ds\\
      &\le \sgm \|\xi\| \alpha^+_{t,\w}\prts{(1+N)d_1(h,g)
            + M d_2(\phi,\psi) }
   \end{align*}
   and this implies
      $$ d_1\prts{J(h,\phi),J(g,\psi)}
         \le \sgm   \prts{(1+N)d_1(h,g) + M d_2(\phi,\psi)}.$$
   On the other hand, using~\ref{eq:split5},~\eqref{ine:hat_gamma<=...} and~\eqref{def:tau} we have
   \begin{align*}
      & \|L\prts{h,\phi}_\w(\xi) - L\prts{g,\psi}_\w(\xi)\|\\
      &\le \int_0^{+\infty} \|\phiw{s} Q_{\theta^s\w} \|
         \, \hat\gamma_{\theta^s\w}(\xi) \ds\\
       &\le  \int_0^{+\infty} \alpha^-_{s,\theta^s\w}\Lip(f_{\theta^s\w})  \|\xi\|\alpha^+_{s,\w} \prts{(1+N)d_1(h,g)
            + M d_2(\phi,\psi) }\ds\\
      & \le \tau \|\xi\|
         \prts{(1+N)d_1(h,g) + M d_2(\phi,\psi)}
   \end{align*}
   and from this estimates it follows that
      $$ d_2\prts{L(h,\phi),L(g,\psi)}
         \le \tau \prts{(1+N)d_1(h,g) + M d_2(\phi,\psi)}.$$
   Hence
   \begin{equation*}
      \begin{split}
         d\prts{T(h,\phi),T(g,\psi)}
         & \le \prts{\sgm + \tau} \prts{(1+N)d_1(x,g)+ M d_2(\phi,\psi)} \\
         & \le \prts{\sgm + \tau} \maxs{1+N,M} d((h,\phi),(g,\psi))
      \end{split}
   \end{equation*}
   and since $\sgm+\tau < 1/2$, $N<1$ and $M<2$, $T$ is a contraction.
\end{proof}

We are finally in conditions to prove Theorem~\ref{thm:global}. Since $\cX^{(B)}_{M,N}$ is a complete metric space and $T$ is a contraction, by Banach Fixed Point Theorem, $T$ has a unique fixed point $(h,\phi)$. Clearly, this fixed point satisfies conditions~\eqref{eq:dyn-split3a} and~\eqref{eq:phi_n}. By Lemma~\ref{lemma:equiv} $(h,\phi)$ also satisfies condition~\eqref{eq:dyn-split3b}. Therefore, by~\eqref{eq:dyn-split1a} and~\eqref{eq:dyn-split1b}, $\prts{h_{t,\w}(\xi), \phi_{\theta^t\w}(h_{t,\w}(\xi))}$ is the trajectory solution of~\eqref{eq:Psi:cont} satisfying the initial condition $(\xi,\phi_\w(\xi)) \in E_\w \x F_\w$, and the graphs $\cV_\w$ are invariant manifolds of~\eqref{eq:Psi:cont}. Moreover, for each $(t,\w,\xi), (t,\w,\bar{\xi})\in \cH$ it follows from~\eqref{eq:||phi_w(xi)-pi_w(bxi)||<=...},~\eqref{ine:||h_n,w(xi)-h_n,w(bxi)||<=M...} and~\eqref{def:M+N} that
\begin{align*}
   & \|\Psi_{\w}^t(\xi,\phi_\w(\xi)) - \Psi_{\w}^t(\bxi,\phi_\w(\bxi))\|\\
   & = \|\prts{h_{t,\w}(\xi),\phi_{\theta^t\w}(h_{t,\w}(\xi))}
      - \prts{h_{t,\w}(\bxi),\phi_{\theta^t\w}(h_{t,\w}(\bxi))}\|\\
   & \le \|h_{t,\w}(\xi) - h_{t,\w}(\bar\xi)\|
      + \|\phi_{\theta^t\w}(h_{t,\w}(\xi)) - \phi_{\theta^t\w}(h_{t,\w}(\bar\xi))\|\\
   & \le (1+N) \|h_{t,\w}(\xi) - h_{t,\w}(\bar\xi)\|\\
   & \le M(1+N) \alpha^+_{t,\w} \|\xi-\bar \xi\|
\end{align*}
and this proves~\eqref{thm:ineq:norm:F_mn(xi...)-F_mn(barxi...)} with $C=M(1+N)$.

\section{Discrete time} \label{sec:disc_time}

Throughout this section we consider the discrete time case $\T = \Z$ and set $\N_0 = \T^+$.

\subsection{Main theorem (discrete time)}\label{section:main:disc}

Let $\Theta \equiv (\W,\Sgm,\mu,\theta)$ be a measure-preserving dynamical system and let $X$ be a Banach space. Consider a measurable linear RDS $\Phi$ on $X$ over $\Theta$ that admits a generalized dichotomy with bounds $\alpha^+$ and $\alpha^-$ and let $f \in \sF$. We will be interested on the RDS $\Psi\colon \N_0\x\W\x X\to X$ satisfying
\begin{equation}\label{eq:Psi}
   \Psi_\w^n(x)=
   \Phi_{\w}^n x + \dsum_{k=0}^{n-1} \Phi_{\theta^{k+1}\w}^{n-k-1} f_{\theta^k\w}(\Psi_{\w}^k(x))
\end{equation}
for all $n \in \N_0$, $x\in X$ and $\w \in \W$. The RDS $\Psi$ can be regarded as the ``solution'' of the random nonlinear difference equation
\begin{equation*}\label{eq:dyn}
      x_{n+1}=\Phi_{\theta^n\w}^1x_n+f_{\theta^n\w}(x_n).
\end{equation*}

Define
\begin{equation*}
   \sgm
   = \sup\limits_{(n,\w)\in \N \x \W}  \dfrac{1}{\alpha^+_{n,\w}}
      \dsum_{k=0}^{n-1} \alpha^+_{n-k-1,\theta^{k+1}\w} \Lip(f_{\theta^k\w}) \alpha^+_{k,\w}
\end{equation*}
and
\begin{equation*}
   \tau
   = \sup_{\w \in \W} \dsum_{k=0}^{+\infty}
      \alpha^-_{k+1,\theta^{k+1}\w} \Lip(f_{\theta^k\w}) \alpha^+_{k,\w}.
\end{equation*}

\begin{theorem} \label{thm:global:disc}
   Consider a measurable linear RDS $\Phi$ on a Banach space $X$ over a measure pre\-ser\-ving dynamical system $\Theta \equiv (\W,\Sgm,\mu,\theta)$ admitting a generalized dichotomy with bounds $\alpha^+$ and $\alpha^-$ and let $f \in \sF$. If
   \begin{equation}\label{eq:CondicaoTeo:disc}
      \lim_{n \to +\infty} \alpha^+_{n,\w} \alpha^-_{n,\theta^n \w} = 0
      \ \ \ \text{ for all } \w \in \W,
   \end{equation}
   and
   \begin{equation*}
      \sgm +  \tau < \frac12,
   \end{equation*}
   then there exist $N \in \ ]0,1[$ and a unique $\phi \in \cL_N$ such that the solution $\Psi$ of~\eqref{eq:Psi} satisfies
   \begin{equation}\label{eq:thm:global:invariance:discrete}
      \Psi_{\w}^n(\cV_{\phi,\w})
      \subseteq \cV_{\phi,\theta^n\w}\ \ \ \text{ for all } (n,\w) \in \N_0 \x \W.
   \end{equation}
   Furthermore, there is a constant $C \in\ ]0,4[$ depending on $\sgm$ and $\tau$ such that
   \begin{equation*}
      \|\Psi_{\w}^n(\xi,\phi_\w(\xi)) - \Psi_{\w}^n(\bxi,\phi_\w(\bxi))\|
      \le C \alpha^+_{n,\w} \, \|\xi - \bxi\|
   \end{equation*}
   for every $(n, \w, \xi), (n, \w, \bxi) \in \cH$.
\end{theorem}

The proof of Theorem~\ref{thm:global:disc} is given in Subsection~\ref{Proof of Theorem discrete}.

\subsection{Corollaries}
We give now some corollaries to Theorem~\ref{thm:global:disc}. Throughout this subsection we consider a real number $\delta \in\ ]0,1/4[$ and a random variable $G \colon \W \to\ ]0,+\infty[$ such that $\dsum_{k=-\infty}^{+\infty} G(\theta^k\w) \le 1$ for all $\w \in \W$.

\begin{corollary}[{Tempered exponential dichotomies}]\label{corollary:tempered:disc}
   Consider a measurable linear RDS $\Phi$ on a Banach space $X$ over a metric dynamical system $\Theta \equiv (\W,\Sgm,\mu,\theta)$ admitting a tempered exponential dichotomy with bounds
   \begin{equation*}
      \alpha^+_{n,\w} = K(\w) \e^{a(\w)n}
      \ \ \ \text{ and  } \ \ \
      \alpha^-_{n,\theta^n\w} = K(\theta^n\w) \e^{b(\w)n}
   \end{equation*}
   such that $a(\w)+b(\w) < 0$ for all $\w \in \W$. Consider a $\theta$-invariant random variable $\gamma(\w)>0$ satisfying $a(\w) + b(\w)+\gamma(\w)<0$ and let $f \in \sF$. If
   \begin{equation*}
      \Lip(f_\w)
      \le \dfrac\delta{K(\theta \w)}\mins{\e^{a(\w)} G(\w),
         \e^{b(\w)}\dfrac{1-\e^{a(\w)+b(\w)+\gamma(\w)}}{\lbd_{K,\gamma(\w),\w}}}
   \end{equation*}
   for all $\w \in \W$, then the same conclusions of Theorem~\ref{thm:global:disc} hold.
\end{corollary}

\begin{proof}
   In the discrete setting $\T=\Z$, conditions~\eqref{eq:tempered} and~\eqref{eq:tempered:lim} are equivalent. Since $K$ is a tempered random variable and $a(\w)+b(\w)<0$, condition~\eqref{eq:CondicaoTeo:disc} holds since it is equivalent to
   \begin{equation*}
      \lim_{n \to +\infty} K(\theta^n \w) \e^{\prts{a(\w) + b(\w)}n} = 0
         \ \ \ \text{ for all } \w \in \W.
   \end{equation*}

   From
   \begin{align*}
      \dfrac{1}{\alpha^+_{n,\w}} \dsum_{k=0}^{n-1}
         \alpha^+_{n-k-1,\theta^{k+1}\w} \Lip(f_{\theta^k\w}) \alpha^+_{k,\w}
      & = \dsum_{k=0}^{n-1}
         K(\theta^{k+1} \w) \e^{-a(\w)} \Lip(f_{\theta^k \w})\\
      & \le \delta \dsum_{k=-\infty}^{+\infty} G(\theta^k\w)\\
      & \le \delta,
   \end{align*}
   we conclude that $\sgm \le \delta$. On the other hand, since $K(\w)\le \e^{\gamma(\w) k}\lbd_{K,\gamma(\w),\theta^k\w}$ for all $\w \in \W$, we have
   \begin{align*}
      &\dsum_{k=0}^{+\infty} \alpha^-_{k+1,\theta^{k+1}\w} \Lip(f_{\theta^k\w}) \alpha^+_{k,\w}\\
      & =  \dsum_{k=0}^{+\infty}
         K(\w)K(\theta^{k+1}\w)\e^{(a(\w)+b(\w))k}\e^{b(\w)}\Lip(f_{\theta^k\w})\\
      & \le \delta(1-\e^{a(\w)+b(\w)+\gamma(\w)}) \dsum_{k=0}^{+\infty}
         \e^{(a(\w)+b(\w)+\gamma(\w))k}\\
      & \le \delta,
   \end{align*}
   and thus it follows that $\tau \le \delta$. Therefore $\sigma+\tau \le 2\delta < 1/2$ and we are in conditions to apply Theorem~\ref{thm:global:disc}.
\end{proof}

The proofs of the next corollaries are similar to the continuous case as illustrated in the proof of Corollary~\ref{corollary:tempered:disc}, and will be omitted. Examples for this type of dichotomies can be found in Example~\ref{ex:2}.

\begin{corollary}\label{corollary:tempered:integral:disc}
   Let $a, b \colon \W \to\ ]0,+\infty[$ and $K \colon \W \to [1,+\infty[$ be random variables and let $\Phi$ be a measurable linear RDS on a Banach space $X$ over a measurable dynamical system $\Theta \equiv (\W,\Sgm,\mu,\theta)$ admitting a dichotomy with bounds \begin{equation*}
      \alpha^+_{n,\w} = K(\w) \e^{S_a(n,\w)}
      \ \ \ \text{ and  } \ \ \
      \alpha^-_{n,\theta^n\w} = K(\theta^n \w) \e^{S_b(n,\w)}
   \end{equation*}
   such that $K(\w) \e^{a(\w)+b(\w)} \le K(\theta \w)$.
   Let $f \in \sF$ be such that
    \begin{equation*}
      \Lip(f_\w)
      \le \delta \mins{\dfrac{\e^{a(\w)}}{K(\theta \w)} G(\w),
         \prts{\dfrac{1}{K(\w)} - \dfrac{\e^{a(\w)+b(\w)}}{K(\theta \w)}}\dfrac{1}{K(\theta\w)}\e^{-b(\w)}}.
   \end{equation*}
   If
   \begin{equation*}\label{eq:corol3:disc}
      \dlim_{n \to +\infty} K(\theta^n \w) \e^{S_{a+b}(n,\w)} = 0
   \end{equation*}
   for all $\w \in \W$, then the same conclusion of Theorem~\ref{thm:global:disc} holds.
\end{corollary}

Consider now  measurable maps $a, b \colon \W \to \R_0^+$, $K \colon \W \to [1,+\infty[$ and a linear measurable RDS $\Phi$ admitting a generalized dichotomy with bounds $\alpha^+$ and $\alpha^-$ given by
\begin{equation}\label{eq:dich:Ka:Kb}
   \alpha^+_{n,\w}
   = K(\w) \dfrac{a(\w)}{a(\theta^n\w)}
   \ \ \ \text{ and } \ \ \
   \alpha^-_{n,\theta^n\w}
   = K(\theta^n\w) \dfrac{b(\w)}{b(\theta^n\w)}.
\end{equation}
For each $\w\in\W$ we define the function $H_\w\colon\Z\to\R$  by
  \[
      H_\w(n)=-\dfrac{1}{a(\theta^n\w)b(\theta^n\w)K(\theta^n\w)}.
  \]
Clearly, $n\in\Z$ we have $H_{\theta^n\w}(k)=H_\w(n+k)$ for all $k,n \in Z$ and all $\w \in \W$.

\begin{corollary}\label{corollary 1:disc}
   Let $a, b \colon \W \to\ ]0,+\infty[$ and $K \colon \W \to [1,+\infty[$ be random variables and let $\Phi$ be a measurable linear RDS on a Banach space $X$ over a measurable dynamical system $\Theta \equiv (\W,\Sgm,\mu,\theta)$ admitting a dichotomy with bounds as in~\eqref{eq:dich:Ka:Kb} and such that
   the map $H_\w$ is an non-decreasing function for each $\w\in\W$. Let $f \in \sF$ satifying
   \begin{equation}\label{eq:corol:f:discrete}
      \Lip(f_\w)
      \le \delta \mins{\frac{a(\w)b(\theta\w)}{K(\theta\w)} (H_{\theta\w}(0)-H_\w(0)), G(\w)} \ \ \ \text{ for all } \w \in \W.
   \end{equation}
   If
   \begin{equation}\label{eq:example:lim d(t)/a(t)b(t)=0:discrete}
      \dlim_ {n\to+\infty}
         \frac{K(\theta^n\w)}{a(\theta^n\w)b(\theta^n\w)}=0,
   \end{equation}
   then the same conclusions of Theorem~\ref{thm:global} hold.
\end{corollary}

\subsection{Proof of Theorem~\ref{thm:global:disc}}\label{Proof of Theorem discrete}

The proof of Theorem~\ref{thm:global:disc} is similar to the proof the continuous time (Theorem~\ref{thm:global}) and therefore we only give a sketch of the necessary adaptations. Fix $M$ and $N$ as in Lemma~\ref{lemma:M&N}. Given $\w\in \W$ and $v_\w=(\xi_\w,\eta_\w)\in E_\w \times F_\w$,
using~\eqref{eq:Psi}, it follows that for each $n\in\N_0$, the trajectory
$\prts{v_{\theta^n\w}}_{n}$, with $v_{\theta^n\w}=\prts{x_{\theta^n\w},y_{\theta^n\w}}\in E_{\theta^n\w}\x F_{\theta^n\w}$, satisfies the following equations
\begin{align}
   x_{\theta^n\w} & = \Phi_{\w}^n \xi_\w + \sum_{k=0}^{n-1} \Phi_{\theta^{k+1}\w}^{n-k-1} P_{\theta^{k+1}\w} f_{\theta^k\w}(x_{\theta^k\w},y_{\theta^k\w}),
      \label{eq:dyn-split1a:discrete}\\
   y_{\theta^n\w}  &= \Phi_{\w}^n \eta_\w +\sum_{k=0}^{n-1} \Phi_{\theta^{k+1}\w}^{n-k-1} Q_{\theta^{k+1}\w} f_{\theta^{k}\w}(x_{\theta^{k}\w},y_{\theta^{k}\w}).
      \label{eq:dyn-split1b:discrete}
\end{align}
In view of the forward invariance required in \eqref{eq:thm:global:invariance:discrete}, each
trajectory of~\eqref{eq:Psi} starting in $\cV_{\phi,\w}$ must be in
$\cV_{\phi,\theta^n\w}$ for every $n \in \N_0$, and thus the
equations~\eqref{eq:dyn-split1a:discrete} and~\eqref{eq:dyn-split1b:discrete} can be written in
the form
\begin{align}
   x_{\theta^n\w} & = \Phi_{\w}^n \xi_\w + \sum_{k=0}^{n-1} \Phi_{\theta^{k+1}\w}^{n-k-1} P_{\theta^{k+1}\w} f_{\theta^k\w}(x_{\theta^k\w},\phi_{\theta^k\w}(x_{\theta^{k}\w})),
      \label{eq:dyn-split2a:discrete}\\
   \phi_{\theta^n\w}(x_{\theta^n\w})  &= \Phi_{\w}^n  \phi_{\w}(\xi_\w) +\sum_{k=0}^{n-1} \Phi_{\theta^{k+1}\w}^{n-k-1} Q_{\theta^{k+1}\w} f_{\theta^{k}\w}(x_{\theta^{k}\w},\phi_{\theta^k\w}(x_{\theta^{k}\w})).
      \label{eq:dyn-split2b:discrete}
\end{align}
Once again, to prove that equations~\eqref{eq:dyn-split2a:discrete} and~\eqref{eq:dyn-split2b:discrete} have
solutions we will use Banach Fixed Point Theorem.

We rewrite conditions~\eqref{eq:dyn-split2a:discrete}
and~\eqref{eq:dyn-split2b:discrete} as in Lemma~\ref{lemma:equiv}.

\begin{lemma} \label{lemma:equiv:discrete}
   Consider $(h,\phi) \in \cX_{M,N}$ such that for every $\w \in \W$, $n \in \N_0$ and $\xi \in E_\w$
   \begin{align}
   h_{n,\w}(\xi) & = \Phi_{\w}^n \xi + \sum_{k=0}^{n-1} \Phi_{\theta^{k+1}\w}^{n-k-1} P_{\theta^{k+1}\w} f_{\theta^{k}\w}(h_{k,\w}(\xi),\phi_{\theta^{k}\w}(h_{k,\w}(\xi))).
      \label{eq:dyn-split3a:discrete}
   \end{align}
   Then the following properties are equivalent:
   \begin{enumerate}[\lb=$\alph*)$,\lm=5mm]
      \item for every $\w \in \W$, $n \in \N_0$, and $\xi \in E_\w$
        \begin{align}
            \phi_{\theta^n\w}(h_{n,\w}(\xi))  &= \Phi_{\w}^n  \phi_{\w}(\xi) +\sum_{k=0}^{n-1} \Phi_{\theta^{k+1}\w}^{n-k-1} Q_{\theta^{k+1}\w} f_{\theta^{k}\w}(h_{k,\w}(\xi),\phi_{\theta^{k}\w}(h_{k,\w}(\xi)));
            \label{eq:dyn-split3b:discrete}
        \end{align}
      \item for every $\w \in \W$, $n \in \N_0$, and $\xi \in E_\w$
        \begin{equation} \label{eq:phi_n:discrete}
            \phi_{\w}(\xi)
               = -\sum_{k=0}^{+\infty} \Phi_{\w}^{-(k+1)}Q_{\theta^{k+1}\w} f_{\theta^{k}\w}(h_{k,\w}(\xi),\phi_{\theta^{k}\w}(h_{k,\w}(\xi))).
        \end{equation}
   \end{enumerate}
\end{lemma}

Let $J$ be the operator that assigns to every $(h,\phi)\in\cX_{M,N}$ the function $J(h,\phi)\colon \cH \to X$ defined by
\begin{align*}
   \prtsr{J\prts{h,\phi}}(n,\w,\xi)
   & = \Phi_\w^n \xi + \dsum_{k=0}^{n-1}
      \Phi_{\theta^{k+1}\w}^{n-k-1} P_{\theta^{k+1}\w}
      f_{\theta^k\w}(h_{k,\w}(\xi),\phi_{\theta^k\w}(h_{k,\w}(\xi)))
\end{align*}
and $L$ be the operator that assigns to every $(h,\phi)\in\cX_{M,N}$ the function $L(h,\phi)\colon \cG \to X$ defined by
\begin{align*}
   \prtsr{L\prts{h,\phi}}(\w,\xi)
   & = -\sum_{k=0}^{+\infty} \Phi_{\w}^{-(k+1)} Q_{\theta^{k+1}\w} f_{\theta^{k}\w}(h_{k,\w}(\xi),\phi_{\theta^k\w}(h_{k,\w}(\xi))).
\end{align*}

We define now the operator $T \colon \cX_{M,N}  \to \cX_{M,N} $ by
   $$ T(h,\phi) = \prts{J(h,\phi), L(h,\phi)}.$$
Similar to Lemma~\ref{lemma:contraction} we have that the operator $T \colon \cX_{M,N} \to \cX_{M,N}$ is a contraction. Thus, since $\cX_{M,N}$ is a complete metric space, by Banach Fixed Point Theorem, $T$ as a unique fixed point $(h,\phi)$. Clearly, this fixed point satisfies conditions~\eqref{eq:dyn-split3a:discrete} and~\eqref{eq:phi_n:discrete}. By Lemma~\ref{lemma:equiv:discrete} $(h,\phi)$ also satisfies condition~\eqref{eq:dyn-split3b:discrete}. Hence, by~\eqref{eq:dyn-split1a:discrete} and~\eqref{eq:dyn-split1b:discrete}, $\prts{h_{n,\w}(\xi), \phi_{\theta^n\w}(h_{n,\w}(\xi))}$ is the orbit of $(\xi,\phi_\w(\xi)) \in E_\w \x F_\w$ by $\Psi$ given at~\eqref{eq:Psi}, and the graphs of $\cV_\w$ are invariant manifolds of~\eqref{eq:Psi}. Moreover, for each $\w\in \W$, $n\in\N_0$ and   $\xi,\bar{\xi}\in E_\w$ it follows from~\eqref{eq:||phi_w(xi)-pi_w(bxi)||<=...}, \eqref{ine:||h_n,w(xi)-h_n,w(bxi)||<=M...} and~\eqref{def:M+N} that
   \begin{equation*}
      \|\Psi_{\w}^n(\xi,\phi_\w(\xi)) - \Psi_{\w}^n(\bxi,\phi_\w(\bxi))\|
         \le M(1+N) \alpha^+_{n,\w} \|\xi-\bar \xi\|,
   \end{equation*}
finishing the proof of  Theorem~\ref{thm:global:disc} with $C=M(1+N)$.
\section*{Acknowledgements}
This work was partially supported by Funda\c c\~ao para a Ci\^encia e Tecnologia through Centro de Matem\'atica e Aplica\c c\~oes da Universidade da Beira Interior (CMA-UBI), project UIDB/MAT/00212/2020.
\bibliographystyle{amsplain}
\providecommand{\bysame}{\leavevmode\hbox to3em{\hrulefill}\thinspace}
\providecommand{\MR}{\relax\ifhmode\unskip\space\fi MR }
\providecommand{\MRhref}[2]{%
  \href{http://www.ams.org/mathscinet-getitem?mr=#1}{#2}
}
\providecommand{\href}[2]{#2}

\end{document}